\documentclass[twoside,12pt,reqno]{amsart}
\usepackage{bbm}
\usepackage{amssymb}
\usepackage{fullpage}

\makeatletter

\newtheorem{Theorem}{Theorem}[section]
\newtheorem{Lemma}[Theorem]{Lemma}
\newtheorem{Proposition}[Theorem]{Proposition}
\newtheorem{Corollary}[Theorem]{Corollary}

\theoremstyle{remark}
\newtheorem{Remark}[Theorem]{Remark}

\numberwithin{equation}{section}

\def\bar{\overline}
\def\sub{\subseteq}
\def\iso{\cong}

\def\onto{\twoheadrightarrow}
\def\isoto{\overset{\sim}{\longrightarrow}}

\def\epsilon{\varepsilon}

\def\kk{{\mathbbm k}}

\def\ba{\text{\boldmath$a$}}
\def\bb{\text{\boldmath$b$}}
\def\bc{\text{\boldmath$c$}}
\def\bd{\text{\boldmath$d$}}
\def\be{\text{\boldmath$e$}}
\def\bp{\text{\boldmath$p$}}

\def\rn{{\mathrm n}}
\def\rs{{\mathrm s}}

\newcommand{\Tab}{\operatorname{Tab}}

\def\diag{\operatorname{diag}}
\def\lmod{\!\operatorname{-mod}}
\def\Lie{\operatorname{Lie}}
\def\tw{{\operatorname{tw}}}
\def\pr{{\operatorname{pr}}}

\def\ab{{\operatorname{ab}}}
\def\ad{\operatorname{ad}}

\def\Ann{\operatorname{Ann}}
\def\col{\operatorname{col}}
\def\gr{\operatorname{gr}}
\def\Mat{\operatorname{Mat}}
\def\row{\operatorname{row}}
\def\sspan{\operatorname{span}}

\def\C{{\mathbb C}}
\def\F{{\mathbb F}}
\def\Z{{\mathbb Z}}

\def\GL{\mathrm{GL}}

\def\b{\mathfrak b}

\def\g{{\mathfrak g}}
\def\gl{\mathfrak{gl}}
\def\h{\mathfrak h}

\def\l{\mathfrak{l}}
\def\m{\mathfrak m}

\def\p{\mathfrak p}
\def\q{\mathfrak{q}}
\def\sl{\mathfrak{sl}}
\def\t{\mathfrak t}
\def\u{\mathfrak{u}}

\def\a{\mathfrak{a}}

\def\cF{\mathcal{F}}

\def\te{\tilde{e}}

\title{\boldmath Minimal dimensional representations of reduced enveloping algebras for $\gl_n$}
\author{Simon M.~Goodwin and Lewis Topley}
\address{School of Mathematics,
University of Birmingham,
Birmingham, B15 2TT,
UK}
\email{s.m.goodwin@bham.ac.uk}
\address{School of Mathematics, Statistics and Actuarial Science, University of Kent,
Canterbury, Kent CT2 7FS, UK}
\email{L.Topley@kent.ac.uk}


\begin{document}

\begin{abstract}
Let $\g = \gl_N(\kk)$, where $\kk$ is an
algebraically closed field of characteristic $p > 0$, and $N \in \Z_{\ge 1}$.
Let $\chi \in \g^*$ and denote by $U_\chi(\g)$ the corresponding
reduced enveloping algebra.  The Kac--Weisfeiler conjecture, which was proved
by Premet, asserts that any finite dimensional $U_\chi(\g)$-module has dimension
divisible by $p^{d_\chi}$, where $d_\chi$ is half the dimension of the coadjoint
orbit of $\chi$.  Our main theorem gives a classification of  $U_\chi(\g)$-modules of dimension $p^{d_\chi}$.
As a consequence, we deduce that they are all parabolically induced from a 1-dimensional
module for $U_0(\h)$ for a certain Levi subalgebra $\h$ of $\g$; we view this as a modular
analogue of M{\oe}glin's theorem on completely primitive ideals in $U(\gl_N(\C))$.
To obtain these results, we reduce to the case
$\chi$ is nilpotent, and then classify the 1-dimensional modules for the corresponding
restricted $W$-algebra.
\end{abstract}

\maketitle

\section{Introduction}

Let $\kk$ be an algebraically closed field of characteristic $p>0$, and $N \in \Z_{\ge 1}$.
Let $G := \GL_N(\kk)$ and $\g := \gl_N(\kk) = \Lie G$.
For $x \in \g$, we write $x^{[p]}$ for the $p$th power of $x$ as a matrix,
and recall that $x \mapsto x^{[p]}$ is the $p$-power map for the restricted Lie algebra
structure on $\g$.  Also we write $x^p$ for the $p$th power
of $x$ in the universal enveloping algebra $U(\g)$ of $\g$.
Then the elements $x^p-x^{[p]}$ are central elements in $U(\g)$, and
the $p$-centre $Z_p(\g)$ of $U(\g)$ is defined to be the subalgebra generated by
$\{x^p - x^{[p]} \mid x \in \g\}$.  It is well-known that $Z_p(\g)$ is $G$-isomorphic to the Frobenius twist
$S(\g)^{(1)}$ of the symmetric algebra on $\g$ and that $U(\g)$ is free of rank $p^{\dim \g}$
over $Z_p(\g)$.

For an irreducible $U(\g)$-module $M$, the central elements $x^p-x^{[p]}$
act on $M$ as $\chi(x)^p$ for some $\chi \in \g^*$, thanks to Quillen's lemma.
We define the ideal $J_\chi$ of $U(\g)$ to be generated
by $\{x^p - x^{[p]} - \chi(x)^p \mid x \in \g\}$, and the reduced
enveloping algebra associated to $\chi$ to be $U_\chi(\g) := U(\g)/J_\chi$.
Then we have seen that any irreducible $U(\g)$-module factors through
$U_\chi(\g)$ for some $\chi \in \g^*$, and that $\dim U_\chi(\g) = p^{\dim \g}$.

Reduced enveloping algebras $U_\chi(\g)$, are defined more generally
for the Lie algebra $\g$ of a reductive algebraic group $G$ over $\kk$,
and their representation theory attracted a great deal of research interest
from leading mathematicians including Friedlander--Parshall, Humphreys,
Jantzen, Kac and Premet in the late 20th century, we refer to the survey articles
\cite{JaLA} and \cite{Hu} for an overview.  There has been continued interest
and progress in the representation theory of reduced enveloping algebra, a notable
advance being the proof by Bezrukavnikov--Mirkovic in \cite{BM} of a conjecture
of Lusztig regarding irreducible modules, for $p$ sufficiently large.
An important conjecture of Kac--Weisfeiler stated in \cite{KWirr} asserts that, for $G$ simple, the dimension of
a $U_\chi(\g)$-module has dimension divisible by $p^{d_\chi}$, where $d_\chi$ is
half the dimension of the coadjoint orbit of $\chi$, and was
proved by Premet in \cite[Theorem I]{PrKW} (under some mild restrictions on $G$ and $p$).
The case $\g = \gl_N(\kk)$ can be deduced directly if $p \nmid N$; also
the case $p \mid N$ can now be obtained from an alternative proof by
Premet in \cite[\S2.6]{PrST}.
We also mention that Friedlander--Parshall
previously proved the conjecture for $\g = \sl_N(\kk)$ and $p \nmid N$ in \cite[Theorem~5.1]{FPind}.
Consequently $p^{d_\chi}$ is smallest dimension of a $U_\chi(\g)$-module,
so we refer to $p^{d_\chi}$-dimensional modules for $U_\chi(\g)$ as {\em minimal dimensional
modules}.  We note that there is a relatively straightforward way, via parabolic
induction, to construct minimal dimensional $U_\chi(\g)$-modules for  $\g = \gl_N(\kk)$, as was
first observed by Friedlander--Parshall in \cite[Corollary~5.2]{FPind}.

In this paper we classify the minimal dimensional $U_\chi(\g)$-modules (for $\g = \gl_N(\kk)$),
as stated in Theorem~\ref{T:main}.  As a consequence we
show that they all can be obtained by parabolically inducing a 1-dimensional
$U_0(\h)$-module for a certain Levi subalgebra $\h$ of $\g$, as
stated in Corollary~\ref{C:moeglin}.
Both of these results are formulated for $\chi \in \g^*$
nilpotent, but as explained in \S\ref{ss:gl} there is a reduction to this case.
We note that Corollary~\ref{C:moeglin}
can be viewed as a modular analogue of M{\oe}glin's theorem, from \cite{Moe} on completely prime ideals
of $U(\gl_N(\C))$.  Further, we remark some of our methods adapt those of Brundan
in \cite{BrM}, in which he gives an alternative proof of M{\oe}glin's theorem.

We require some notation to state our main results.  This is all set out in detail
in Section~\ref{S:prelims}, and here we only point out the necessary parts for the statements of
Theorem~\ref{T:main} and Corollary~\ref{C:moeglin}.

The trace form on $\g := \gl_N(\kk)$ is denoted $(\cdot\,,\cdot)$ and allows us
to identify $\g \iso \g^*$.  Consequently, we can talk about Jordan decomposition
of elements of $\g^*$ and nilpotent elements of $\g^*$.  We let $\b$ be the Borel
subalgebra of upper triangular matrices and $\t$ the maximal toral subalgebra of
diagonal matrices.

Let $\bp = (p_1 \le p_2 \le \dots \le p_n)$ be a partition of $N$, and let
$\pi$ be a pyramid associated to $\bp$.   This means that $\pi$ is diagram with $N$ boxes
organised in rows, with row lengths given by $\bp$ as defined in \S\ref{ss:pyramid};
further the boxes in $\pi$ are labelled from 1 to $N$ along rows starting from the top row.
There is some choice of the pyramid $\pi$, and much of the notation
below is dependant on this choice; as the results are all valid for any choice of $\pi$, we choose
to work in this generality, and just note that the left justified pyramid is one choice that can be made.

From $\pi$ we define the nilpotent element $e \in \g$ as in \eqref{e:enilp}, and
let $\chi := (e,\cdot) \in \g^*$.
Then $\chi \in \g^*$ is nilpotent and
as we range over all partitions $\bp$ of $N$, we get representatives
of all coadjoint $G$-orbits of nilpotent elements of $\g^*$.
As explained in \S\ref{ss:pyramid}, we have that $\chi$ is in standard Levi form with respect to $\b$.
A good grading of $\g$ for $e$ is defined in \eqref{e:goodgrading}, and from
this we can define the parabolic subalgebra $\p$ with Levi factor $\h$ as in
\eqref{e:phandm}.
We note that $\t$ is contained in $\p$ and $\h$, but that $\b \nsubseteq \p$.

Let $\F_p \sub \kk$ denote the field of $p$ elements.
We define $\Tab_{\kk}(\pi)$ to be fillings
of the boxes of $\pi$ with elements from $\kk$; and we
define $\Tab_{\F_p}(\pi) \sub \Tab_{\kk}(\pi)$ to be the filling
with elements from $\F_p$.
We refer to elements of $\Tab_{\kk}(\pi)$ as $\pi$-tableau.
Given $A \in \Tab_{\kk}(\pi)$, we denote the entry in the box labelled $i$ in $\pi$
by $a_i$.  Let $\epsilon_1,\dots,\epsilon_N$ be the standard basis of $\t^*$ and
define $\lambda_A = \sum_{i=1}^N a_i \epsilon_i \in \t^*$ for $A \in \Tab_{\kk}(\pi)$.
We say that $A \in \Tab_{\kk}(\pi)$ is {\em column connected} if $a_i = a_j+1$ whenever
box $i$ is directly above box $j$ in $\pi$.

The weight
$\rho \in \t^*$ is defined in \eqref{e:rho}: it is a convenient renormalization
of the half sum of positive roots
corresponding to $\b$.
Also we define $\bar \rho \in \t^*$ in  \eqref{e:barrho}, which
is the half sum of positive roots for a Borel subalgebra
of $\p$, with a convenient renormalization.

We recap some established representation theory of $U_\chi(\g)$ and interpret it in our notation,
see for example \cite[Section 10]{JaLA}, more detail is given in \S\ref{ss:modules}.
Since $e \in \b$, we have that $\chi(\b) = 0$.
Given $A \in \Tab_{\F_p}(\pi)$ we define $\kk_A$ to be the 1-dimensional $U_0(\b)$-module
on which $\t$ acts via $\lambda_A - \rho$, and the baby Verma module to be
$Z_\chi(A) = U_\chi(\g) \otimes_{U_0(\b)} \kk_A$.  It is known that
$Z_\chi(A)$ has a unique maximal submodule, and we denote the simple head of $Z_\chi(A)$ by
$L_\chi(A)$.  Further, any
irreducible $U_\chi(\g)$-module is isomorphic to $L_\chi(A)$ for some $A \in \Tab_{\F_p}(\pi)$, and
for $A,A' \in \Tab_{\F_p}(\pi)$,
we have $L_\chi(A) \iso L_\chi(A')$ if and only if $A$ is row equivalent to $A'$. We recall that we
say that $A$ is row equivalent to $A'$ if we can obtain $A'$ from $A$ by reordering the entries in rows.

We are now in a position to state our main theorem giving a classification of
minimal dimensional $U_\chi(\g)$-modules.

\begin{Theorem} \label{T:main}
Let $\g = \gl_N(\kk)$, let $\pi$ be a pyramid corresponding to a partition
$\bp$ of $N$,
and let $\chi$ be the nilpotent element of $\g^*$ determined
by $\pi$. For $A \in \Tab_{\F_p}(\pi)$, we have that $L_\chi(A)$ is a minimal
dimensional $U_\chi(\g)$-module if and only if $A$ is row equivalent to a column connected
$\pi$-tableau.
\end{Theorem}

To state Corollary~\ref{C:moeglin}, we have to define certain 1-dimensional
$U_0(\p)$-modules.  As explained in \S\ref{ss:modules}, given $A \in \Tab_{\F_p}(\pi)$, we have that $\lambda_A - \bar \rho$
is the weight of a 1-dimensional $U_0(\h)$-module if and only $A$ is column connected.
For column connected $A \in \Tab_{\F_p}(\pi)$, we define $\bar \kk_A$ to be the one dimensional $U_0(\p)$-module
obtained by inflating the 1-dimensional $U_0(\h)$-module with weight $\lambda_A - \bar \rho$.  In
Theorem~\ref{T:1dsimplespara}, we show that
$L_\chi(A) \iso U_\chi(\g) \otimes_{U_0(\p)} \bar \kk_A$ for column connected $A \in \Tab_{\F_p}(\pi)$.
Combining this with Theorem~\ref{T:main}, we immediately deduce.

\begin{Corollary} \label{C:moeglin}
Let $\g = \gl_N(\kk)$, let $\pi$ be a pyramid corresponding to a partition
$\bp$ of $N$,
let $\chi$ be the nilpotent element of $\g^*$ determined
by $\pi$, and let $\p$ be the parabolic subalgebra
of $\g$ determined by $\pi$.
Let $L$ be a minimal dimensional $U_\chi(\g)$-module.  Then $L \iso U_\chi(\g) \otimes_{U_0(\p)} \bar \kk_A$ for some
column connected $A \in \Tab_{\F_p}(\pi)$.
\end{Corollary}

We give an outline of the main ideas in the proof of Theorem~\ref{T:main}.
The key step is to rephrase the problem in terms of $W$-algebras through
Premet's equivalence.  Let $U(\g,e)$ be the finite $W$-algebra as in \cite[Definition 4.3]{GT};
in fact we use an equivalent definition in this paper as a subalgebra of $U(\p)$
as explained in \S\ref{ss:walg}.
The restricted $W$-algebra $U_0(\g,e)$ is as in \cite[Definition 8.5]{GT}, though as explained in \S\ref{ss:walg}
our notation in this paper differs from that in \cite{GT} and we view $U_0(\g,e)$ as a subalgebra of $U_0(\p)$.
The definitions of these $W$-algebras in \cite{GT} are inspired by work of Premet,
where $U(\g,e)$ has appeared for $p$ sufficiently large and is obtained from
a characteristic 0 finite $W$-algebra via reduction modulo $p$, see for example \cite[\S2.5]{PrCQ}.

We recall that Premet's equivalence, which is stated in Theorem~\ref{T:premequiv},
 gives an equivalence of categories between $U_\chi(\g)\lmod$ and $U_0(\g,e)\lmod$.
Moreover, through this equivalence a $U_0(\g,e)$-module of dimension $m$ corresponds to
a $U_\chi(\g)$-module of dimension $mp^{d_\chi}$.  Therefore, in order to prove Theorem
\ref{T:main}, we want to classify the 1-dimensional $U_0(\g,e)$-modules.

In fact we classify all 1-dimensional $U(\g,e)$-modules and determine
which ones factor through the quotient map $U(\g,e) \onto U_0(\g,e)$.  We show that
$U(\g,e)$ is a modular truncated shifted Yangian, see Theorem~\ref{T:gensandPBW}.
This is proved by following the methods of Brundan--Kleshchev in \cite{BKshift},
but now using the PBW theorem for $U(\g,e)$ given in \cite[Theorem 7.3]{GT} and reduction modulo $p$ arguments.
In particular, this allows us to determine the abelianisation $U(\g,e)^{\ab}$ of $U(\g,e)$,
by observing that a calculation by Premet from \cite[Theorem 3.3]{PrCQ} applies in characteristic $p$.
As mentioned above we view $U(\g,e)$ as a subalgebra of $U(\p)$.
Thus we obtain 1-dimensional $U(\g,e)$-modules by restricting 1-dimensional $U(\p)$-modules.
Rather than using the labelling of 1-dimensional $U(\p)$-modules
as $\bar \kk_A$ for column connected $A$ in $\Tab_\kk(\pi)$ above,
we in fact consider $U(\p)$-modules $\widetilde \kk_A$, where a different shift is used.
Using the description of $U(\g,e)^{\ab}$, we are able to
deduce that the restriction of the modules $\widetilde \kk_A$ for column connected $A \in \Tab_\kk(\pi)$ give all
of the 1-dimensional $U(\g,e)$-modules.
Moreover, for column connected $A,A' \in \Tab_\kk(\pi)$
we deduce that the restrictions of $\widetilde \kk_A$ and $\widetilde \kk_{A'}$ are isomorphic
if and only if $A$ is row equivalent to $A'$.  We denote $\widetilde \kk_A$ restricted to $U(\g,e)$
by $\widetilde \kk_{\bar A}$.  Our methods for this classification of 1-dimensional $U(\g,e)$-modules
are similar to those used by Brundan in \cite[Section 2]{BrM}.

Our next step is to show, for column connected $A \in \Tab_\kk(\pi)$, that
$\widetilde \kk_{\bar A}$ factors to a module for $U_0(\g,e)$ if and only
if $A \in \Tab_{\F_p}(\pi)$.  This deduction is not immediate and is given in Theorem~\ref{T:1dsforrestw}.
From here we are in a position to apply Premet's equivalence to determine the minimal dimensional
$U_\chi(\g)$-modules.  A key step for this is given by
Theorem~\ref{T:1dsimplespara}, which says that $L_\chi(A) \iso U_\chi(\g) \otimes_{U_0(\p)} \bar \kk_A$
for column connected $A \in \Tab_{\F_p}(\pi)$.
This requires us to identify a vector in
$U_\chi(\g) \otimes_{U_0(\p)} \bar \kk_A$, which spans a 1-dimensional
$U(\b)$-module with weight $\lambda_A - \rho$.
From this we can deduce that $L_\chi(A)$ is minimal dimensional if $A$ is column connected.
By applying our classification of 1-dimensional $U_0(\g,e)$-modules and Premet's equivalence, we are
thus able to conclude that the set $L_\chi(A)$ for $A \in \Tab_{\F_p}(\pi)$ column connected (up to row equivalence)
gives all of the minimal dimensional $U_\chi(\g)$-modules, which proves Theorem~\ref{T:main}.
In fact it is possible to show that
through Premet's equivalence $\widetilde \kk_{\bar A}$ corresponds to $L_\chi(A)$;
this is discussed in Remark~\ref{R:corresp}.

We end the introduction with some remarks about minimal dimensional
modules for reduced enveloping algebras $U_\chi(\g)$ for $\g$ the Lie
algebra of a reductive algebraic group over $\kk$.
The assertion that there is a $U_\chi(\g)$-module of dimension $p^{d_\chi}$ is now known as Humphreys' conjecture,
see \cite[\S8]{Hu}, though we note that the question was asked earlier by Kac in his
review of \cite{PrKW} on the {\em Mathematical Reviews}.  There has been lots of progress
on this conjecture recently and thanks to the results of Premet in \cite{PrMF}
it is now known to be true for $p$ sufficiently large; further Premet states that in
forthcoming work he will give an explicit lower bound on $p$.  The questions of whether
the minimal dimensional modules can be classified, and whether they are parabolically induced
are also of great interest.  We plan to consider these in future work, and note that
the characteristic 0 version of the latter is addressed in work of Premet and the second author
in \cite{PT}.

\subsection*{Acknowledgments}
Both authors would like to thank the University of Padova and the Erwin
Schr{\"o}dinger Institute, Vienna, where parts of this work were carried out.
The first author is supported in part by EPSRC grant EP/R018952/1.
The second author gratefully acknowledges funding from the European Commission, Seventh
Framework Programme, Grant Agreement 600376, as well as EPSRC grant EP/N034449/1.
We thank Alexander Premet for helpful correspondence about this work, and the referee
for useful comments.

\section{Preliminaries} \label{S:prelims}

\subsection{The general linear Lie algebra and reduced enveloping algebras} \label{ss:gl}

Let $\kk$ be an algebraically closed field of characteristic $p > 0$ and let $N \in \Z_{\ge 1}$.
Throughout this paper $G := \GL_N(\kk)$ and $\g := \gl_N(\kk)$ is the
Lie algebra of $G$, which is spanned
by the matrix units $\{e_{i,j} \mid 1 \le i,j \le N\}$.
Let $(\cdot\,,\cdot) : \g \times \g \to \kk$ denote the
trace form associated to the natural representation of $G$,
which we use to identify $\g \iso \g^*$ as $G$-modules.
The universal enveloping algebra of $\g$ is denoted $U(\g)$.

We occasionally need to call on some results from characteristic zero and so we fix some
more notation. We let $\g_\Z$ denote the general linear Lie $\Z$-algebra $\gl_N(\Z)$ and we write
$\g_\C$ for $\gl_N(\C)$. Throughout we use the identifications $\g \iso \g_\Z \otimes_\Z \kk$ and
$\g_\C \iso \g_\Z \otimes_\Z \C$, and by a slight abuse of notation we view the matrix units $e_{i,j}$
as elements of $\g_\Z$ or $\g_\C$ when it is convenient to do so.
We often consider subalgebras of $\g$, which are spanned by matrix units, so have analogues inside
$\g_\Z$ and $\g_\C$ and we denote them by decorating with subscripts $\Z$ and $\C$.
We mention that
since $\g_\Z$ is a free $\Z$-module the PBW theorem holds for $U(\g_\Z)$, so that $U(\g_\Z)$ is a free $\Z$-module
with a basis consisting of ordered monomials in the matrix units with respect to any choice of total order.

Let $g \in G$, $x \in \g$ and $\chi \in \g^*$.  We write $g \cdot x$ for the image of
$x$ under the adjoint action of $g$, so as matrices $g \cdot x = gxg^{-1}$; this
action extends to an action on $U(\g)$ by algebra automorphisms.
The centralizer of $x$ in $G$
is denoted $G^x := \{g \in G \mid g \cdot x = x\}$ and the centralizer
of $x$ in $\g$ is denoted $\g^x := \{y \in \g \mid [y,x] = 0\}$;
we note that we have $\g^x = \Lie(G^x)$.
We write $G \cdot \chi$ for the coadjoint
orbit of $\chi$.  It is well-known that $\dim(G \cdot \chi)$ is even
and we define $d_\chi := \frac{1}{2} \dim(G \cdot \chi)$.

Let $T \subseteq B \subseteq G$ be the maximal torus and Borel subgroup
consisting of diagonal matrices and upper triangular matrices respectively, and
let $\t := \Lie(T)$, $\b := \Lie(B)$.  We use the notation $\diag(d_1,\dots,d_n)$
to denote the element of $T$ with $d_i$ in the $i$th entry of the diagonal.
We write $X^*(T)$ for the group of characters, and let
$\{\epsilon_1,\dots,\epsilon_N\}$ be the standard basis of $X^*(T)$
defined by $\epsilon_i(\diag(d_1,\dots,d_n)) = d_i$.
Let $\Phi \sub X^*(T)$ be the root system of
$G$ with respect to $T$,
so $\Phi = \{\epsilon_i - \epsilon_j \mid 1 \le i,j \le n, i \ne j\}$.
We write $e_{i,j}$ for the matrix unit that spans the root space corresponding to $\epsilon_i -\epsilon_j$.
The root subgroup corresponding to $\epsilon_i -\epsilon_j$ is the image of $u_{i,j} : \kk \to G$
defined by $u_{i,j}(s) := 1 + se_{i,j}$, and the adjoint action of $u_{i,j}(s)$ on $e_{k,l}$
is given by the formula
\begin{equation} \label{e:adjoint}
u_{i,j}(s) \cdot e_{k,l} = e_{k,l} + s \delta_{j,k}e_{i,l} - s\delta_{l,i}e_{k,j} - s^2 \delta_{j,k}\delta_{l,i} e_{i,j}.
\end{equation}
Where it is convenient we allow ourselves to view a character $\alpha \in X^*(T)$ as an
element of $\t^*$ by writing
$\alpha$ for $d\alpha : \t \to \kk$; this is a slight abuse of notation,
because $d\alpha = 0$ for any $\alpha \in pX^*(T)$.

There is a natural restricted structure on $\g$,
where the $p$-power map $x \mapsto x^{[p]}$ is
given by taking the $p$th power of $x$ as a matrix.  In particular, we note that $e_{i,j}^{[p]} = \delta_{i,j}e_{i,j}$
for $1 \le i,j \le N$.
The $p$-centre of $U(\g)$ is the subalgebra of the centre of $U(\g)$ generated by
$\{e_{i,j}^p - e_{i,j}^{[p]} \mid 1\leq i,j \leq N\}$.  It follows from the PBW theorem that
$U(\g)$ is a free $Z_p(\g)$-module of rank $p^{\dim \g}$. Further, there is a natural identification
$Z_p(\g) \iso \kk[(\g^*)^{(1)}]$, where $(\g^*)^{(1)}$ denotes the Frobenius twist of $\g^*$.
Given $\chi \in \g^*$ we define $J_\chi$ to be the ideal of $U(\g)$ generated by
$\{x^p-x^{[p]} - \chi(x)^p \mid x \in \g\}$, and the {\em reduced enveloping
algebra corresponding to $\chi$} to be $U_\chi(\g) := U(\g)/J_\chi$.

As stated in the introduction, the Kac--Weisfeiler conjecture, which is a
theorem of Premet, states that $p^{d_\chi}$ is a factor of the dimension
of any $U_\chi(\g)$-module.  We refer to $U_\chi(\g)$-modules of dimension $p^{d_\chi}$
as {\em minimal dimensional modules}, and note that such modules are clearly irreducible.

Let $\chi \in \g^*$.  There is unique $x \in \g$ such that $\chi = (x,\cdot)$.
We have a Jordan decomposition $x = x_\rs + x_\rn$
of $x$, and thus a corresponding decomposition $\chi = \chi_\rs + \chi_\rn$.
We say that $\chi$ is nilpotent if $\chi = \chi_\rn$. Next we recall the ``reduction'' to
the case $\chi$ nilpotent in the representation theory of $U_\chi(\g)$ from \cite[Section 3]{FPmod};
as is noted in \cite[Section 8]{FPmod}, this reduction can also be deduced
from \cite[Theorem~2]{KWirr}.
Let $\l = \g^{x_\rs}$, let $\q$ be a parabolic subalgebra of $\g$ with Levi factor
$\l$ and let $\u$ denote the nilradical of $\q$. We can parabolically induce a $U_\chi(\l)$-module $M$, to obtain the $U_\chi(\g)$-module
$U_\chi(\g) \otimes_{U_\chi(\q)} M$, where $M$ is the $U_\chi(\q)$-module on which $\u$ acts trivially.
This gives a functor $U_\chi(\l)\lmod \to U_\chi(\g)\lmod$ and it is proved in \cite[Theorem~3.2]{FPmod} that
this is an equivalence of categories; in turn there is an equivalence
$U_\chi(\l)\lmod \iso U_{\chi_\rn}(\l)\lmod$ as follows from \cite[Corollary~3.3]{FPmod}.
Further, the theory of Jordan normal forms implies that
$\dim(G \cdot \chi) = \dim(L \cdot \chi_\rn) + 2\dim \u$.
Therefore, through the equivalence of categories $U_{\chi_\rn}(\l)\lmod \iso U_\chi(\g)\lmod$,
minimal dimensional modules for $U_{\chi_\rn}(\l)$ correspond to minimal
dimensional modules for $U_\chi(\g)\lmod$.  This justifies our restriction
to nilpotent $\chi$ in the statements of Theorem~\ref{T:main} and Corollary~\ref{C:moeglin}.

\subsection{Pyramids} \label{ss:pyramid}

We require the combinatorics of pyramids to set up some notation.
For more details on this we refer to \cite[Section 7]{BKshift}.

We fix a partition $\bp = (p_1,\dots,p_n)$ on $N$ with $p_1 \le \cdots \le p_n$.
A pyramid $\pi$ associated to $\bp$ is a diagram with $p_n$ boxes in the bottom row,
$p_{n-1}$ boxes in the row above it, and so forth, stacked in such a way that
every box which is not in the bottom row lies directly above a box in the row beneath it,
and boxes occur consecutively in each row.  The boxes in the pyramid are
numbered along rows from left to right and from top to bottom.
For example, the pyramids associated to the partition $\bp = (2,5)$ are
\begin{equation}
\label{e:somepyramids}
\begin{array}{c}
\begin{picture}(60,24) \put(0,0){\line(1,0){60}}
\put(0,12){\line(1,0){60}} \put(0,24){\line(1,0){24}}
\put(0,0){\line(0,1){24}} \put(12,0){\line(0,1){24}}
\put(24,0){\line(0,1){24}} \put(36,0){\line(0,1){12}}
\put(48,0){\line(0,1){12}}
\put(60,0){\line(0,1){12}}
\put(3,14.5){\hbox{1}}
\put(15,14.5){\hbox{2}}
\put(3,2.5){\hbox{3}}
\put(15,2.5){\hbox{4}}
\put(27,2.5){\hbox{5}}
\put(39,2.5){\hbox{6}}
\put(51,2.5){\hbox{7}}
\end{picture}
\end{array}
,\:\:\:\:
\begin{array}{c}
\begin{picture}(60,24) \put(0,0){\line(1,0){60}}
\put(0,12){\line(1,0){60}} \put(12,24){\line(1,0){24}}
\put(0,0){\line(0,1){12}} \put(12,0){\line(0,1){24}}
\put(24,0){\line(0,1){24}} \put(36,0){\line(0,1){24}}
\put(48,0){\line(0,1){12}}
\put(60,0){\line(0,1){12}}
\put(15,14.5){\hbox{1}}
\put(27,14.5){\hbox{2}}
\put(3,2.5){\hbox{3}}
\put(15,2.5){\hbox{4}}
\put(27,2.5){\hbox{5}}
\put(39,2.5){\hbox{6}}
\put(51,2.5){\hbox{7}}
\end{picture}
\end{array},\:\:\:\:
\begin{array}{c}
\begin{picture}(60,24) \put(0,0){\line(1,0){60}}
\put(0,12){\line(1,0){60}} \put(24,24){\line(1,0){24}}
\put(0,0){\line(0,1){12}} \put(12,0){\line(0,1){12}}
\put(24,0){\line(0,1){24}} \put(36,0){\line(0,1){24}}
\put(48,0){\line(0,1){24}}
\put(60,0){\line(0,1){12}}
\put(27,14.5){\hbox{1}}
\put(39,14.5){\hbox{2}}
\put(3,2.5){\hbox{3}}
\put(15,2.5){\hbox{4}}
\put(27,2.5){\hbox{5}}
\put(39,2.5){\hbox{6}}
\put(51,2.5){\hbox{7}}
\end{picture}\end{array}
\:\:\:\:\text{and}\:\:\:\:
\begin{array}{c}
\begin{picture}(60,24) \put(0,0){\line(1,0){60}}
\put(0,12){\line(1,0){60}} \put(36,24){\line(1,0){24}}
\put(0,0){\line(0,1){12}} \put(12,0){\line(0,1){12}}
\put(24,0){\line(0,1){12}} \put(36,0){\line(0,1){24}}
\put(48,0){\line(0,1){24}}
\put(60,0){\line(0,1){24}}
\put(39,14.5){\hbox{1}}
\put(51,14.5){\hbox{2}}
\put(3,2.5){\hbox{3}}
\put(15,2.5){\hbox{4}}
\put(27,2.5){\hbox{5}}
\put(39,2.5){\hbox{6}}
\put(51,2.5){\hbox{7}}
\end{picture}
\end{array}.
\end{equation}
Let $l = p_n$. The columns of $\pi$ are labelled $1,2,...,l$ from left to right and the rows are
labelled $1,2,...,n$ from top to bottom.  We denote the heights of the columns in $\pi$ by $q_1, q_2,...,q_l$.
The box in $\pi$ containing
$i$ is referred to as the $i$th box, and we write $\row(i)$ and $\col(i)$ for the row and column of the $i$th box
respectively.

We fix a pyramid $\pi$ corresponding to $\bp$ for the rest of this paper.
From $\pi$, we define the {\em shift matrix} $\sigma = (s_{i,j})$ as follows.
For $1 \le i < j \le n$ we let $s_{j,i}$ be the left indentation of the $i$th row of $\pi$
relative to the $j$th row, and we let $s_{i,j}$ be the right indentation of the $i$th row of $\pi$
relative to the $j$th row; also we set $s_{i,i} = 0$.  For example the shift matrices
associated to the pyramids in \eqref{e:somepyramids} are
\begin{equation*}
\label{e:someshifts}
\left(\begin{array}{cc} 0 & 3 \\ 0 & 0 \end{array}\right),
\left(\begin{array}{cc} 0 & 2 \\ 1 & 0 \end{array}\right),
\left(\begin{array}{cc} 0 & 1 \\ 2 & 0 \end{array}\right) \text{ and }
\left(\begin{array}{cc} 0 & 0 \\ 3 & 0 \end{array}\right).
\end{equation*}

\subsection{The nilpotent element and subalgebras}
We define the nilpotent element
\begin{equation} \label{e:enilp}
e := \sum_{\substack{\row(i) = \row(j) \\ \col(i) = \col(j) - 1}} e_{i,j} \in \g.
\end{equation}
For example for each of the pyramids in \eqref{e:somepyramids}, we have
$e = e_{1,2}+e_{3,4}+e_{4,5}+e_{5,6}+e_{6,7}$.
Observe that $e$ has Jordan blocks of size $p_1, p_2,....,p_n$.
We define $\chi := (e,\cdot) \in \g^*$.  We also note that $\chi$ is in standard Levi form
(in the sense of \cite[Definition 3.1]{FPdef}) with respect to
the simple roots corresponding to the Borel subalgebra $\b$.

The first part of the following lemma gives a basis of $\g^e$, and
can be verified by observing that the proof of \cite[Lemma 7.3]{BKshift} is also
valid in positive characteristic.  The second part of the lemma is verified
by direct calculation.

\begin{Lemma}\label{L:centraliserbasis}
Let
$$
c_{i,j}^{(r)} :=
\sum_{\substack{1\leq h,k,\leq N\\ \row(h) = i, \row(k) = j \\ \col(k) - \col(h) + 1 = r}} e_{h,k}
$$
for $0 \leq i,j \leq n$ and $r > s_{i,j}$.
\begin{itemize}
\item[(a)]  The centralizer $\g^e$ of $e$ in $\g$ has basis
$$
\{c_{i,j}^{(r)}  \mid  0 \leq i,j \leq n, s_{i,j} < r \leq s_{i,j} + p_{\min(i,j)}\},
$$
\item[(b)]  We have
$$
[c_{i,j}^{(r)}, c_{k, l}^{(s)}] = \delta_{j,k} c_{i,l}^{(r+s-1)} - \delta_{i,l} c_{k,j}^{(r+s-1)}.
$$
\end{itemize}
\end{Lemma}

Consider the cocharacter $\mu : \kk^\times \to T \sub G$ defined by
$\mu(t) = \diag(t^{\col(1)},\dots,t^{\col(n)})$.
Using $\mu$ we define the $\Z$-grading
\begin{equation} \label{e:goodgrading}
\g = \bigoplus_{k \in \Z} \g(k) \quad \text{where} \quad \g(k) := \{x \in \g \mid \mu(t) x = t^k x \text{ for all } t\in \kk^\times\} .
\end{equation}
Since the adjoint action of $\mu(t)$ on a matrix unit is
given by $\mu(t) \cdot e_{i,j} = t^{\col(j) - \col(i)} e_{i,j}$,
we have $\g(k) = \sspan\{e_{i,j} \mid \col(j) - \col(i) = k\}$.
From the classification of good gradings in \cite[Section 4]{EK}, we see that the grading in \eqref{e:goodgrading}
is a good grading for $e$.  In fact to get a good grading we should scale the
grading by a factor of 2, as we have $e \in \g(1)$.  We refer also to \cite[Section 3]{GT} where good gradings
are considered in positive characteristic, and it is shown that the ``same classification''
of good gradings holds.
Since the grading in \eqref{e:goodgrading} is good we have that $\g^e \sub \bigoplus_{k \ge 0} \g(k)$, which
can also be seen directly from Lemma~\ref{L:centraliserbasis},
Now it follows from \cite[Theorem 1.4]{EK} that $\dim \g^e = \dim \g(0)$;
this can also be verified directly from the basis given in Lemma~\ref{L:centraliserbasis}.

We define the following subalgebras of $\g$
\begin{equation} \label{e:phandm}
\p :=  \bigoplus_{k \geq 0} \g(k), \quad \h :=  \g(0) \quad \text{and} \quad
\m := \bigoplus_{k < 0} \g(k).
\end{equation}
Then $\p$ is a parabolic subalgebra of $\g$, and $\h$ is the Levi factor of $\p$ containing $\t$.
Further, $\m$ is the nilradical of the opposite parabolic to $\p$.
We recall that the heights of the columns in $\pi$ are $q_1, q_2,...,q_l$, and we see that $\h$
is isomorphic to $\gl_{q_1}(\kk) \oplus \gl_{q_2}(\kk) \oplus \cdots \oplus \gl_{q_l}(\kk)$.
Also $\m$ is the Lie algebra of the closed subgroup
$M$ of $G$ generated by the root subgroups $u_{i,j}(\kk)$
with $\col(j) < \col(i)$.

We recall that $d_\chi$ denotes half the dimension of the coadjoint $G$-orbit
of $\chi$.  So we also have that $d_\chi$ is half the dimension of the adjoint $G$-orbit of $e$,
and thus we see that $d_\chi := \dim \m$, because $\dim \g^e = \dim \g(0)$.

\subsection{Tableaux and weights}

We require various weights in $\t^*$, which are used as shifts and to label
certain modules.  These weights can be encoded by fillings of $\pi$ as we explain below,
then we move on to give the weights we need.

A {\em $\pi$-tableau} is a diagram obtained by filling the boxes of $\pi$ with elements of $\kk$.
The set of all tableau of shape $\pi$ is denoted $\Tab_\kk(\pi)$, and we write
$\Tab_{\F_p}(\pi) \sub \Tab_\kk(\pi)$ for those tableaux with entries
in $\F_p$. For $A \in \Tab_\kk(\pi)$, we write $a_i$
for the entry in the $i$th box of $A$.
Two tableaux
are called {\em row-equivalent} if one can be obtained from the other
by permuting the entries in the rows.
A tableau $A \in \Tab_\kk(\pi)$ is {\em column-connected} if whenever the $j$th box of $\pi$ is directly below the $i$th box
we have $a_i = a_j + 1$.

For $A \in \Tab_\kk(\pi)$
we define a weight $\lambda_A \in \t^*$ by
\begin{equation*}
\lambda_A := \sum_{i=1}^N a_i \epsilon_i.
\end{equation*}

To understand the required weights it helps for us to give a decomposition of $\Phi$.  We
define
\begin{align*}
\Phi_+ &:= \{\epsilon_i - \epsilon_j \in \Phi \mid \row(i) < \row(j)\}, \\
\Phi_0 &:= \{\epsilon_i - \epsilon_j \in \Phi \mid \row(i) = \row(j)\} \text{ and } \\
\Phi_- &:= \{\epsilon_i - \epsilon_j \in \Phi \mid \row(i) > \row(j)\}.
\end{align*}
Also we define
\begin{align*}
\Phi(+) &:= \{\epsilon_i - \epsilon_j \in \Phi \mid \col(i) < \col(j)\}, \\
\Phi(0) &:= \{\epsilon_i - \epsilon_j \in \Phi \mid \col(i) = \col(j)\} \text{ and } \\
\Phi(-) &:= \{\epsilon_i - \epsilon_j \in \Phi \mid \col(i) > \col(j)\}.
\end{align*}
Then for $\eta,\xi \in \{-,0,+\}$, we define
$$
\Phi(\eta)_\xi = \Phi(\eta) \cap \Phi_\xi.
$$
We note that $\Phi(0)_0 = \varnothing$ and that $\Phi_+ \cup \Phi(+)_0$ is the system of positive
roots corresponding to $\b$.
Further, $\Phi(+) \cup  \Phi(0)_+$ is
the system of positive roots corresponding to a Borel subalgebra
contained in $\p$, and $\Phi(-) \cup \Phi(0)_+$ is another system of
positive roots.

Having set up this notation we are in a position to give
the weights that we require.
First we define
\begin{equation} \label{e:rho}
\rho := - \sum_{i=1}^N i \epsilon_i
\end{equation}
this is a shifted half sum of positive roots for $\b$, and is given by
$$
\rho =  \frac{1}{2} \left(\sum_{\alpha \in \Phi_+ \cup \Phi(+)_0} \alpha\right) - \delta,
$$
where
$$
\delta = \frac{N+1}{2} \sum_{i=1}^N \epsilon_i.
$$
We note that we should be careful in the above formulas when $p = 2$, though as the final
value of $\rho$ only involves integer coefficients this is not a problem.

We also require a ``choice of $\rho$'' corresponding to the system of
positive roots $\Phi(+) \cup  \Phi(0)_+$, and we define
\begin{equation} \label{e:barrho}
\bar \rho :=    \frac{1}{2} \left(\sum_{\alpha \in \Phi(+) \cup \Phi(0)_+} \alpha\right) - \delta
\end{equation}
More explicitly, we have
$$
\bar \rho = - \sum_{i=1}^N ((q_1+\dots+q_{\col(i)-1}) + \row(i) - (n-q_{\col(i)})) \epsilon_i.
$$
The weight
$$
\gamma := \sum_{\alpha \in \Phi(-)_+} \alpha,
$$
is important for Theorem~\ref{T:1dsimplespara}, because
\begin{eqnarray}
\label{e:rhoandgamma}
\rho = \bar \rho + \gamma.
\end{eqnarray}
We define
\begin{equation} \label{e:eta}
\eta := \sum_{i=1}^N (n - q_{\col(i)} - \dots - \cdots - q_l) \epsilon_i,
\end{equation}
and
$$
\rho_\h := - \sum_{i=1}^N \row(i) \epsilon_i,
$$
which is a shifted choice $\rho$ for the Borel subalgebra $\b \cap \h$ of $\h$.
Further, we define
$$
\beta := \sum_{i=1}^N (((q_1+\dots+q_{\col(i)-1})-(q_{\col(i)+1}+\dots+q_l)) \epsilon_i = \sum_{\alpha \in \Phi(-)} \alpha.
$$
and
$$
\widetilde \rho := \bar \rho + \beta.
$$
We note that $\widetilde \rho$ is a shifted choice of $\rho$
for the system of positive roots $\Phi(-) \cup \Phi(0)_+$.
An important identity for us is
\begin{equation} \label{e:tilderho}
\widetilde \rho = \bar \rho + \beta = \eta + \rho_\h = \sum_{\alpha \in \Phi(-) \cup \Phi(0)_+} \alpha.
\end{equation}

\subsection{Some modules for $U_\chi(\g)$} \label{ss:modules}

The weights introduced in the previous subsection are required to
define some modules for $\h$ and for $\g$.  In what follows it is helpful
to note that $\chi|_\p = 0$, so that we can view $U_0(\h) \sub U_0(\p) \sub U_\chi(\g)$.

We note that $\lambda \in \t^*$ is the weight of
1-dimensional $U(\h)$-module if and only if $\lambda(e_{i,i}) = \lambda(e_{j,j})$
whenever $\col(i) = \col(j)$, and also that $\widetilde \rho(e_{i,i}) = \widetilde \rho(e_{j,j}) - 1$,
when the $i$th box in $\pi$ is directly above the $j$th box.  Thus we deduce that, for
$A \in \Tab_\kk(\pi)$, we have $\lambda_A - \widetilde \rho$ is the weight
of a 1-dimensional $U(\h)$-module if and only if $A$ is column connected.
For column connected $A$ we denote this 1-dimensional $U(\h)$-module by $\widetilde \kk_A$.

Similarly, given $A \in \Tab_\kk(\pi)$, we have $\lambda_A - \bar \rho$ is the weight of a
1-dimensional $U(\h)$-module if and only if $A$ is column connected.  In this case
we denote the 1-dimensional $U(\h)$-module by $\bar \kk_A$, and note that it factors to a module
for $U_0(\h)$ if and only if $A \in \Tab_{\F_p}(\pi)$.  For $A \in \Tab_{\F_p}(\pi)$, we can inflate $\bar \kk_A$ to
a $U_0(\p)$-module and consider the induced module $N_\chi(A) := U_\chi(\g) \otimes_{U_0(\p)} \bar \kk_A$.
We have that $N_\chi(A) \iso U_\chi(\m)$ as a $U_\chi(\m)$-module, so that $\dim N_\chi(A)
= p^{\dim \m} = p^{d_\chi}$ and $N_\chi(A)$ is a minimal dimensional $U_\chi(\g)$-module.

Let $A \in \Tab_{\F_p}(\pi)$.  We define $\kk_A$ to be the 1-dimensional $U_0(\b)$-module
where $\t$ acts by $\lambda_A - \rho$, and the nilradical of $\b$ acts trivially.
The {\em baby Verma module} $Z_\chi(A)$ is defined to be $Z_\chi(A) := U_\chi(\g) \otimes_{U_0(\b)} \kk_A$.  Since
$\chi$ is in standard Levi form for the Levi subalgebra $\g_0$ with basis $\{e_{i,j} \mid \row(i) = \row(j)\}$,
$Z_\chi(A)$ has a simple head, which we denote by $L_\chi(A)$;
this essentially follows from the results in \cite[Section 3]{FPdef}, see also
\cite[Proposition 10.7]{JaLA}.  Moreover, we have that $L_\chi(A) \iso L_\chi(A')$ if
and only if $A$ is row equivalent to $A'$, see \cite[Corollary 3.5]{FPdef} or
\cite[Proposition 10.8]{JaLA}.  To see this we note that the shift by $\rho$
in our labelling of the simple modules, transforms the dot action of the
$W_0$ on $\t^*$ in \cite{FPdef} to the standard action, where $W_0$ is the Weyl
group of $\g_0$ with respect to $T$; and then this action
corresponds to permutations of entries in rows of tableau.
Given a $U_\chi(\g)$-module $M$ we say $v \in M$ is a
highest weight vector (for $\b$) of weight $\lambda \in \t^*$ if $\b v \sub \kk v$ and $tv = \lambda(t)v$ for
all $t \in \t$; so if $v \in M$ is a highest weight vector of weight $\lambda_A - \rho$, then there
is a homomorphism $Z_\chi(A) \to M$ sending $1 \otimes 1_A$ to $v$, where $1_A$ denotes the generator
of $\kk_A$.

The following theorem is key to this paper and
gives a compatibility between the modules $L_\chi(A)$ and $N_\chi(A)$.

\begin{Theorem} \label{T:1dsimplespara}
For column connected $A \in \Tab_{\F_p}(\pi)$ we have
$L_\chi(A) \iso N_\chi(A)$ and has dimension $p^{d_\chi}$.  In particular, for column connected
$A,A' \in \Tab_{\F_p}(\pi)$, we have $N_\chi(A) \iso N_\chi(A')$ if and only if
$A$ is row equivalent to $A'$.
\end{Theorem}

\begin{proof}
From the discussion above, we know $N_\chi(A)$ has dimension $p^{d_\chi}$, so it
is a minimal module for $U_\chi(\g)$ and thus simple.
It follows that if we can find a highest weight vector $v \in N_\chi(A)$ for $\b$ of weight $\lambda_A - \rho$,
then $L_\chi(A) \iso N_\chi(A)$ as required.  This can be seen by noting that
the homomorphism $Z_\chi(A) \to N_\chi(A)$ will factor to give this isomorphism.
The claim regarding row equivalence was justified in the remarks
preceding the statement of the theorem.

We observe that the root vectors corresponding to roots in $\Phi(-)_+$ span a $p$-nilpotent subalgebra
$\mathfrak a$ of $\g$.  We let
$$
I = \{(i,j) \mid \epsilon_i - \epsilon_j \in \Phi(-)_+\} =\{ (i,j) \mid \col(i) > \col(j), \row(i) < \row(j)\} \subseteq \{1,...,N\}^2,
$$
so that $\a$ has basis $\{e_{i,j} \mid (i,j) \in I\}$.
Since all of the elements of this basis have nonzero $\t^e$ weight, we
see that the restriction of $\chi$ to $\a$ is zero.
Hence, the restricted enveloping algebra $U_0(\a)$ embeds in $U_\chi(\g)$, and
consequently
$e_{i,j}^p = 0$ in $U_0(\a) \sub U_\chi(\g)$ for $(i,j) \in I$.

There is an action of $T$ on $U_0(\a)$, and
\begin{eqnarray}
\label{e:defineu}
u = \prod_{(i,j) \in I} e_{i,j}^{p-1},
\end{eqnarray}
is in the unique weight space of maximal weight (with respect to the positive roots for $\b$).
Further, this weight space is 1-dimensional, which implies that the product in \eqref{e:defineu} can be
taken in any order (up to rescaling).

Let $\bar 1_A$ denote the generator of $\bar \kk_A$.
Observe that under the adjoint action $\t$ acts on $u$ with weight
$(p-1) \gamma = -\gamma = \rho - \bar\rho$ by
\eqref{e:rhoandgamma}.
Therefore, $v := u \otimes \bar 1_A$ is a weight vector for $\t$ with weight
$\lambda_A - \bar\rho + (\bar \rho - \rho) = \lambda_A - \rho$. In order to complete the proof we
must show $v$ is a highest weight vector for the action of $\b$, which
requires us to show that $e_{i,i+1} v = 0$ for $i=1,..., N-1$.

\smallskip

We first deal with the case where $\row(i) = \row(i+1)$ and we let $r := \row(i)$.
We begin by decomposing $I$ into four subsets:
\begin{align*}
I_1 &:= \{(j,k) \in I \mid \row(j) = r\};\\
I_2 &:= \{(j,k) \in I \mid \row(k) = r\};\\
I_3 &:= \{(j,k) \in I \mid \row(j) < r,  \row(k) > r\}; \text{ and} \\
I_4 &:= \{(j,k) \in I \mid (j,k) \notin I_1 \cup I_2 \cup I_3\}.
\end{align*}
We record three facts about commuting elements which are straightforward to verify directly. \\
{\bf Fact (i).} $e_{i,i+1}$ commutes with $e_{j,k}$ for $(j,k)\in I_3 \cup I_4$. \\
{\bf Fact (ii).} The elements $\{e_{j,k} \mid (j,k) \in I_1 \cup I_3\}$ pairwise commute. \\
{\bf Fact (iii).} The elements $\{e_{j,k} \mid (j,k) \in I_2 \cup I_3\}$ pairwise commute.

For $s = 1,2,3$, we see that $\{e_{j,k} \mid (j,k) \in I_s\}$ is the basis of an abelian subalgebra
of $\a$.  Therefore, the element $u_s := \prod_{(j,k) \in I_s} e_{j,k}^{p-1}$ does not
depend on the order of the product.
We choose an arbitrary ordering of $I_4$ and let $u_4 := \prod_{(j,k) \in I_4} e_{j,k}^{p-1}$.

We proceed with three claims, which we use to show that $e_{i,i+1} v = 0$.

\smallskip
\noindent
\textbf{Claim 1.} $(\ad(e_{i,i+1}) u_1) \otimes \bar 1_A = 0$.

\noindent Observe that $\ad(e_{i,i+1}) u_1$ is a sum of expressions of the form
\begin{equation}
\label{e:1dsimpleseqn1}
u_1^{(i+1, l)} := e_{i,l}(e_{i+1, l}^{p-2}) \prod e_{j,k}^{p-1}.
\end{equation}
where $(i+1, l) \in I_1$, and the product is taken over all $(j,k) \ne (i+1, l) \in I_1$.
Since all matrix units occurring in \eqref{e:1dsimpleseqn1} are of the form $e_{a,b}$ with $\row(a) = r$ and $\row(b) > r$
all of these factors commute so can be reordered.

We consider two cases to complete the proof of Claim 1.  The first case is when $(i,l) \in I_1$.
Then $u_1^{(i+1, l)}$ contains a factor of $e_{i,l}^p$, so that $u_1^{(i+1, l)} = 0$.
The second case is when $\col(i) = \col(l)$ and so $e_{i,l} \in [\h, \h]$.
In this case $e_{i, l} \bar 1_A = 0$ and so
$u_1^{(i+1, l)} \otimes \bar 1_A = 0$.

\smallskip
\noindent
\textbf{Claim 2.} \emph{$u_3 e_{j,k} u_1 \otimes \bar 1_A = 0$ whenever $\col(j) = \col(k)$
and $\row(j) < \row(k) = r$.}

\noindent We have $e_{j,k} \in [\h, \h]$ so $e_{j,k}\bar 1_A =0$.
Thus it suffices to show that $u_3 (\ad(e_{j,k}) u_1) = 0$.
Observe that $\ad(e_{j,k}) u_1$ is a sum of monomials of the form
\begin{eqnarray}
\label{e:1dsimpleseqn2}
e_{j,l} (e_{k,l}^{p-2}) \prod e_{k', l'}^{p-1}
\end{eqnarray}
where $(k,l) \in I_1$ and the product is taken over $(k',l') \ne (k,l) \in I_1$.
Similar to the comments following \eqref{e:1dsimpleseqn1} the matrix units occurring in
\eqref{e:1dsimpleseqn2} all commute and so can be reordered.
Since $\row(j) < r$ and $\row(l) > \row(k) = r$ we have $e_{j,l} \in I_3$.
Applying Fact (iii) above we see that
$u_3 e_{j,l} e_{k,l}^{p-2} \prod e_{k', l'}^{p-1}$ contains a factor of $e_{j,l}^p$, hence is
equal to 0.  This proves Claim 2.

\smallskip
\noindent
\textbf{Claim 3.} $u_3 (\ad(e_{i, i+1}) u_2) u_1 \otimes \bar 1_A = 0$.\\
\noindent Observe that $\ad(e_{i,i+1}) u_2$ is a sum of expressions of the form
\begin{eqnarray}
\label{e:1dsimpleseqn3}
u_2^{(l, i)} := -e_{l,i+1}(e_{l, i}^{p-2}) \prod e_{j',k}^{p-1}.
\end{eqnarray}
where $(l,i) \in I_2$ and the product is taken over all $(j,k) \in I_2$ with $(j, k) \neq (l,i)$.  The
matrix units occurring here all commute, so can be
reordered.
We consider two cases.  The first case is when
$(l,i+1) \in I_2$. Then $u_2^{(l,i)}$ contains a factor of $e_{l,i+1}^p$ and $u_2^{(l,i)} = 0$.
The second case is when $\col(l) = \col(i+1)$.
Then we can use Claim 2 along with Fact (iii) to
show that $u_3 u_2^{(l, i)} u_1 \otimes \bar 1_A = 0$.
This completes the proof of Claim 3.

\smallskip

We now combine these claims to prove that $e_{i,i+1} v = 0$.
Since $e_{i,i+1}$ lies in the nilradical of $\p$ we have that $e_{i,i+1} \bar 1_A = 0$.
Thus it suffices to prove
$\ad(e_{i,i+1})(u_4 u_3 u_2 u_1) \otimes \bar 1_A = 0$. Applying Fact (i) we only need to check
$u_4 u_3 (\ad(e_{i,i+1})u_2) u_1 \otimes \bar 1_A = 0$ and $u_4 u_3 u_2 (\ad(e_{i,i+1})u_1) \otimes \bar 1_A = 0$,
which are given by Claim 3 and Claim 1 respectively.

\smallskip

We move on to deal with the case $\row(i) < \row(i+1)$, and show that $e_{i,i+1}v = 0$.

For this case first suppose that $\col(i) > \col(i+1)$.  Then we have $e_{i,i+1} \in \a$.
By the remarks following \eqref{e:defineu} we can write
$u = e_{i, i+1}^{p-1} u_0$ for some $u_0 \in U_0(\a)$ and so $e_{i, i+1}u = 0$, which implies
that $e_{i,i+1}v = 0$.

The case where $\col(i) = \col(i+1)$, which only happens when $p_{\row(i)} = 1$ and $s_{i+1,i} = 0$.
Then we have $e_{i, i+1} \in [\h, \h]$ and $e_{i,i+1} \bar 1_A = 0$, so we
are just required to show that $[e_{i,i+1}, u] = 0$.  This is done with commutator arguments
similar to those used above, so we omit the details.
\end{proof}

\subsection{The $W$-algebra $U(\g,e)$ and its $p$-centre} \label{ss:walg}
Since $e \in \g(1)$, we have that $\chi$ vanishes on $\g(k)$ for $k \ne -1$.
Therefore, $\chi$ restricts to a character of $\m$.
We define $\m_\chi := \{ x - \chi(x) \mid x\in \m\} \subseteq U(\g)$, which is a
Lie subalgebra of $U(\g)$.
By the PBW theorem there is a direct sum decomposition
$$
U(\g) = U(\g) \m_\chi \oplus U(\p)
$$
We let $\pr : U(\g) \to U(\p)$ be the projection onto the second factor.  Also
we abbreviate and write $I := U(\g)\m_\chi$, and define $Q:= U(\g)/I$.

As explained in \cite[\S4.3]{GT} the adjoint action of $M$ on $U(\g)$
gives an adjoint action of $M$ on $Q$.  In \cite[Definition 4.3]{GT} the
$W$-algebra associated to $e$ is defined to be
\begin{equation*}
\{u + I \in Q \mid g \cdot u + I = u + I \text{ for all } g \in M\}.
\end{equation*}
In this paper, we prefer to work with an equivalent realization of $U(\g,e)$
as a subalgebra of $U(\p)$.  For this we require the {\em twisted adjoint action}
of $M$ on $U(\p)$, which is defined by
\begin{equation*} \label{e:twistM}
\tw(g) \cdot u := \pr (g \cdot u),
\end{equation*}
for $g \in M$ and $u \in U(\p)$.
By using $\pr$ to identify $U(\g)/I$ with $U(\p)$, we can equivalently define the
$W$-algebra associated to $e$ to be the invariant subalgebra
\begin{equation*}
U(\g,e) := U(\p)^{\tw(M)} = \{ u \in U(\p) \mid \tw(g) \cdot u = u \text{ for all } g \in M\}.
\end{equation*}

We want to recast some of the material from \cite[Section 8]{GT} in
our setting where $U(\g,e) = U(\p)^{\tw(M)}$.  We begin with the $p$-centre
of $U(\g,e)$, and to define this we note that the $p$-centre $Z_p(\p)$ of $U(\p)$ is stable under the
twisted adjoint action of $M$ of $U(\p)$.
The {\em $p$-centre of $U(\g,e)$} is defined in \cite[Definition 8.1]{GT}, and in our
setting, it is given by
$$
Z_p(\g,e) := Z_p(\p)^{\tw(M)} \subseteq U(\g,e).
$$

Let $\psi \in \p^* \sub \g^*$.
We write $J_\psi^\p$ for the ideal of $U(\p)$ generated
$\{x^p - x^{[p]} - \psi(x)^p \mid x \in \p\}$,
the {\em reduced $W$-algebra corresponding to $\psi$} as
\begin{equation*}
U_\psi(\g,e) := U(\g,e) / (J_\psi^\p \cap U(\g,e)).
\end{equation*}
We note that our notation here differs from that used in \cite[Definition 8.5]{GT} by a shift
of $\chi$, i.e.\ $U_\psi(\g,e)$ here would be denoted $U_{\chi+\psi}(\g,e)$ there
(to make sense of $\chi+\psi \in \g^*$ we identify $\p^* = \Ann_{\g^*}(\m) \sub \g^*$).
This change in notation is partly justified by the fact that
the kernel of the restriction of the projection $U(\p) \onto U_\psi(\p)$
to $U(\g,e)$ is $J_\psi \cap U(\g,e)$.  Consequently, we can identify
$U_\psi(\g,e)$ with the image of $U(\g,e)$ in $U_\psi(\p)$.

It turns out that for $\psi \ne \psi'$
we can have $J_\psi^\p \cap U(\g,e) = J_{\psi'}^\p \cap U(\g,e)$,
so that $U_\psi(\g,e) = U_{\psi'}(\g,e)$.
To explain precisely when this happens we need to translate some of the material from
\cite[\S8.2]{GT} to our setting.  We write $\m^\perp \sub \g$ for the annihilator
of $\m$ with respect to $(\cdot\,,\cdot)$, and note that we can identify
$\p^* \iso e+\m^\perp$ via $(\cdot\,,\cdot)$.  There is an adjoint action of
$M$ on $e+\m^\perp$, and this translates through the identification
$\p^* \iso e+\m^\perp$ to an
action of $M$ on $\p^*$, which we refer to as the {\em twisted action of $M$ on $\p^*$}.
For $\phi \in \p^*$, $g \in M$ and $x \in \p$ this twisted adjoint action is given
by $(\tw(g) \cdot \phi)(x) = \chi(g^{-1} \cdot x -x) + \phi(g^{-1} \cdot x)$.

Now we state the required part of \cite[Lemma 8.6]{GT} in our notation.

\begin{Lemma} \label{L:samered}
We have that $U_\psi(\g,e) = U_{\psi'}(\g,e)$ if and only if $\psi$
and $\psi'$ are conjugate under the twisted $M$-action on $\p^*$.
\end{Lemma}

Thanks to Quillen's lemma, an irreducible $U(\g,e)$-module $L$ factors to
a module for $U_\psi(\g,e)$ for some $\psi \in \p^*$.  Further, it is
clear from the definitions that, for $\psi,\psi' \in \p^*$,
the module $L$ factors to a module for both $U_\psi(\g,e)$ and for $U_{\psi'}(\g,e)$
if and only if $J_\psi^\p \cap U(\g,e) = J_{\psi'}^\p \cap U(\g,e)$, which by
the previous lemma occurs if only $\psi$
and $\psi'$ are conjugate under the twisted $M$-action.

We also recall Premet's equivalence in Theorem~\ref{T:premequiv} below.
This theorem is based on \cite[Theorem 2.4]{PrST}, and the statement
here can be deduced from \cite[Lemma 2.2(c)]{PrCQ} and \cite[Proposition 8.7 and Lemma 8.8]{GT}, see also
\cite[Remark 9.4]{GT}. For the statement, we view $\psi \in \p^*$ as an element of $\g^*$
via the identification $\p^* = \Ann_{\g^*}(\m) \sub \g^*$.
Also we define $Q^\psi = Q/J_{\chi+\psi}Q$, and recall that as explained in
\cite[\S8.3]{GT} $Q^\psi$ is a left $U_{\chi+\psi}(\g)$-module and a right $U_\psi(\g,e)$-module

\begin{Theorem} \label{T:premequiv}
Let $\psi \in \p^*$.  We have
\begin{itemize}
\item[(a)]
$
U_{\chi+\psi}(\g) \iso \Mat_{p^{d_\chi}} U_\psi(\g,e);
$
\item[(b)] the functor from $U_{\psi}(\g,e)\lmod$ to $U_{\chi+\psi}(\g)\lmod$
given by
\begin{equation} \label{e:premequiv1}
M \mapsto Q^\psi \otimes_{U_\psi(\g,e)} M
\end{equation}
is an equivalence of categories with quasi-inverse given
by
\begin{equation} \label{e:premequiv2}
V \mapsto V^{\m_\chi} := \{v \in V \mid \m_\chi v = 0\}.
\end{equation}
\item[(c)]
$
\dim (Q^\psi \otimes_{U_\psi(\g,e)} M) = p^{d_\chi} \dim M,
$
for a finite dimensional $U_\psi(\g,e)$-module $M$.
\end{itemize}
\end{Theorem}

We also recall that $U(\g,e)$ has a PBW basis, which is described in \cite[Theorem 7.3]{GT}.
We summarize the properties that we require in Proposition~\ref{P:PBW} below and adapt
the statement to the case $\g = \gl_N(\kk)$.  For this we first
have to give some notation.
We fix a basis
$x_1,...,x_r$ of $\g^e$, chosen so that $x_i \in \g(n_i)$, where $n_i \in \Z_{\ge 0}$.
Let $I_\p = \{(i,j) \mid 1 \le i,j, \le N, e_{i,j} \in \p\}$ and fix an order on $I_\p$.
For $\ba = (a_{i,j}) \in \Z_{\geq 0}^{I_\p}$  we write
\begin{equation} \label{e:e^a}
\be^\ba := \prod_{(i,j) \in I_\p} e_{i,j}^{a_{i,j}} \in U(\p),
\end{equation}
and define $|\ba| = \sum_{(i,j) \in I_\p} a_{i,j}$ and $|\ba|_e = \sum_{(i,j) \in I_\p} (\col(j)-\col(i)+1)a_{i,j}$.

We can now state our proposition about the PBW basis of $U(\g,e)$; it is a consequence
of \cite[Lemma 7.1 and Theorem 7.3]{GT}. We remind the reader that the graded degrees in this paper
differ from those in {\it loc. cit.} by a factor of 2.

\begin{Proposition} \label{P:PBW}
$ $
\begin{itemize}
\item[(a)] There are elements $\Theta(x_1),\dots,\Theta(x_r)$ of $U(\g,e)$ of the form
\begin{equation}\label{e:Theta}
\Theta(x_i)
=  x_i + \sum_{|\ba|_e \le n_i + 1} \lambda_{\ba,i} \be^\ba
\end{equation}
where $\lambda_{\ba,i} \in \kk$ satisfy $\lambda_{\ba,i} = 0$ whenever $|\ba|_e = n_i+1$ and $|\ba| = 1$.
\item[(b)] Given any elements $\Theta(x_1),\dots,\Theta(x_r) \in U(\g,e)$ of the form in \eqref{e:Theta} the ordered monomials
in $\Theta(x_1),\dots,\Theta(x_r)$ form a basis of $U(\g,e)$.
\end{itemize}
\end{Proposition}

Let $\Theta(x_i),\Theta(x_j)$ be elements of $U(\g,e)$ of the form \eqref{e:Theta}.
Then a commutator calculation shows that
\begin{equation} \label{e:commTheta}
[\Theta(x_i),\Theta(x_j)] = [x_i,x_j] + \sum_{|\ba|_e \le n_i + n_j + 1} \mu_\ba \be^\ba,
\end{equation}
where $\mu_\ba \in \kk$ satisfy $\mu_\ba = 0$ whenever $|\ba|_e = n_i+n_j+1$ and $|\ba| = 1$.
The key ingredient for this calculation is to observe that if we take the
commutator $[\be^\ba,\be^\bb]$ for $\ba, \bb \in \Z_{\ge 0}^{I_\p}$,
then we get a linear combinations of terms $\be^\bc$ with $|\bc|_e = |\ba|_e+|\bb|_e-1$ and
$|\bc| = |\ba|+|\bb|-1$, plus a linear combination of terms $\be^\bd$ with $|\bd|_e < |\ba|_e+|\bb|_e-1$.

\section{Modular truncated shifted Yangian}

In this section we consider the modular shifted Yangian $Y_n(\sigma)$ and its truncation $Y_{n,l}(\sigma)$.
The algebras $Y_n(\sigma)$ have been studied in recent work of Brundan and the
second author, \cite{BT}.  Here we recall some of the results in {\em loc.\ cit.\ }and
move on to verify that the truncation $Y_{n,l}(\sigma)$ has structure theory similar to that
in characteristic 0.
In the next section we exploit formulas from \cite[Section 9]{BKshift}
to show that the modular finite $W$-algebra $U(\g,e)$ is isomorphic to the shifted truncated Yangian $Y_{n,l}(\sigma)$ of level $l = p_n$.

We recall that $\sigma$ is the shift matrix for the pyramid $\pi$.
The {\em modular shifted Yangian} $Y_n(\sigma)$ is the $\kk$-algebra with generators
\begin{equation}\label{e:Ygens}
\begin{array}{c} \{D_i^{(r)} \mid 1\leq i \leq n, r>0\} \cup \{E_i^{(r)} \mid 1\leq i < n, r> s_{i,i+1} \} \\ \cup \, \{F_i^{(r)} \mid 1\leq i < n, r> s_{i+1,i}\} \end{array}
\end{equation}
and relations
\begin{align}
\big[D_i^{(r)}, D_j^{(s)}\big] &=  0,\label{e:r2}\\
\big[E_i^{(r)},F_j^{(s)}\big] &= -\delta_{i,j}
\sum_{t=0}^{r+s-1} D_{i+1}^{(r+s-1-t)}\widetilde D_{i}^{(t)},\label{e:r3}\\
\big[D_i^{(r)}, E_j^{(s)}\big] &= (\delta_{i,j}-\delta_{i,j+1})
\sum_{t=0}^{r-1} D_i^{(t)} E_j^{(r+s-1-t)},\label{e:r4}\\
\big[D_i^{(r)}, F_j^{(s)}\big] &= (\delta_{i,j+1}-\delta_{i,j})
\sum_{t=0}^{r-1} F_j^{(r+s-1-t)}D_i^{(t)} ,\label{e:r5}\\
\big[E_i^{(r)}, E_i^{(s)}\big] &=
\sum_{t=r}^{s-1} E_i^{(t)} E_i^{(r+s-1-t)}
\hspace{11mm}\text{if $r < s$},\label{e:r6}\\
\big[F_i^{(r)}, F_i^{(s)}\big] &=
\sum_{t=s}^{r-1}
F_i^{(r+s-1-t)} F_i^{(t)}\hspace{11mm}\text{if $r > s$},\label{e:r7}\\
\big[E_i^{(r+1)}, E_{i+1}^{(s)}\big]&-
 \big[E_i^{(r)}, E_{i+1}^{(s+1)}\big]=
E_i^{(r)} E_{i+1}^{(s)},\label{e:r8}\\
\big[F_i^{(r)}, F_{i+1}^{(s+1)}\big]&- \big[F_i^{(r+1)}, F_{i+1}^{(s)}\big] =
 F_{i+1}^{(s)} F_i^{(r)},\label{e:r9}\\
\big[E_i^{(r)}, E_j^{(s)}\big] &= 0 \hspace{37.9mm}\text{ if }|i-j|> 1,\label{e:r10}\\
\big[F_i^{(r)}, F_j^{(s)}\big] &= 0 \hspace{37.9mm}\text{ if }|i-j|>
1,\label{e:r11}
 \end{align}\begin{align}
\Big[E_i^{(r)}, \big[E_i^{(s)}, E_j^{(t)}\big]\Big] &+
\Big[E_i^{(s)}, \big[E_i^{(r)}, E_j^{(t)}\big]\Big] = 0 \quad\text{ if }|i-j|=1, r
\neq s,\label{e:r12}\\
\Big[F_i^{(r)}, \big[F_i^{(s)}, F_j^{(t)}\big]\Big] &+
\Big[F_i^{(s)}, \big[F_i^{(r)}, F_j^{(t)}\big]\Big] = 0 \quad\text{ if
}|i-j|=1,r \neq s\label{e:r13}\\
\Big[E_i^{(r)}, \big[E_i^{(r)}, E_j^{(t)}\big]\Big] &= 0 \hspace{39.2mm}\text{ if } |i - j| = 1, \label{e:rel15}\\
\Big[F_i^{(r)}, \big[F_i^{(r)}, F_j^{(t)}\big]\Big] &= 0 \hspace{39.2mm}\text{ if } |i - j| = 1,\label{e:rel16}
\end{align}
for all admissible $i,j,r,s,t$.
In the relations, the shorthand
$D_i^{(0)} = \widetilde{D}_i^{(0)} := 1$ is used,
and the elements $\widetilde{D}_i^{(r)}$ for $r > 0$
are defined recursively by
$\widetilde{D}_i^{(r)} := -\sum_{t=1}^r D_i^{(t)} \widetilde D_i^{(r-t)}$.

This presentation of $Y_n(\sigma)$ is given in \cite[Theorem~4.15]{BT} and is
modelled on the Drinfeld presentation
of the shifted Yangian defined over $\C$, as introduced in \cite[Section 2]{BKshift}.
It is proved in \cite[Theorem~4.14]{BT} that
there is a PBW basis for $Y_n(\sigma)$, whose description does not depend on the
characteristic $p$.  Before stating this result
it is necessary to introduce some additional elements.
We define
\begin{align*}
E_{i,i+1}^{(r)} &:= E_{i}^{(r)}, \\
F_{i,i+1}^{(r)} &:= F_i^{(r)}
\end{align*}
for $ i =1,\dots,n-1$,
and inductively define
\begin{align} \label{e:Eij}
E_{i,j}^{(r)} := [E_{i, j-1}^{(r - s_{j-1,j})}, E_{j-1}^{(s_{j-1, j} + 1)}] & \text{ for } 1 \leq i < j \leq n \text{ and } r > s_{i,j}, \\
\smallskip \label{e:Fij}
F_{i,j}^{(r)} := [F_{j-1}^{(s_{j,j-1} + 1)}, F_{i,j-1}^{(r-s_{j, j-1})}] & \text{ for } 1 \leq i < j \leq n \text{ and } r > s_{j,i}.
\end{align}
Then \cite[Theorem~4.14]{BT} says that monomials in the elements
\begin{equation} \label{e:pbwgens}
\begin{array}{c}
\{D_i^{(r)} \mid 1 \le i \le n, r > 0 \} \cup \{E_{i,j}^{(r)} \mid 1 \le i < j \le n, r > s_{i,j} \} \\
\cup \, \{F_{i,j}^{(r)} \mid 1 \le i < j \le n, r > s_{j,i}\}
\end{array}
\end{equation}
in any fixed order give a basis of $Y_n(\sigma)$.

The shifted Yangian has the {\em canonical filtration} which we denote
$Y_n(\sigma) = \bigcup_{r \geq 0} \cF_r Y_n(\sigma)$
and is defined by declaring that $D_i^{(r)}, E_{i,j}^{(r)}, F_{i,j}^{(r)} \in \cF_r Y_n(\sigma)$, i.e.\
that $\cF_r Y_n(\sigma)$ is the spanned by the monomials in these elements of total degree $\le r$.
It is immediate from the relations \eqref{e:r2}--\eqref{e:rel15}
that the associated graded algebra $\gr Y_n(\sigma)$ is commutative.

The {\em truncated shifted Yangian of level $l$} is denoted $Y_{n,l}(\sigma)$ and
defined to be the quotient of $Y_n(\sigma)$ by the ideal generated by $\{D_1^{(r)} \mid r> p_1\}$;
this definition is taken from \cite[Section 6]{BKshift} where it is given for characteristic 0.
We recall that $l = p_n$, so that $p_1 = l - s_{1,n} - s_{n,1}$.
The truncated shifted Yangian inherits the canonical filtration from $Y_n(\sigma)$ and we write
$Y_{n,l}(\sigma) = \bigcup_{i \geq 0} \cF_i Y_{n,l}(\sigma)$. The associated
graded algebra $\gr Y_{n,l}(\sigma)$ is certainly commutative, as it is a quotient of $\gr Y_n(\sigma)$.
When working with $Y_{n,l}(\sigma)$ we often abuse notation by using the same
symbols $D_i^{(r)}, E_{i,j}^{(r)}, F_{i,j}^{(r)}$ to refer to the elements of $Y_n(\sigma)$ and
their images in $Y_{n, l}(\sigma)$.

The next lemma gives a spanning set for $Y_{n,l}(\sigma)$ and should be viewed as
a modular version of \cite[Lemma~6.1]{BKshift}; though we note that it
is less general as we do not deal with parabolic presentations here.
We recover the full PBW theorem for $Y_{n,l}(\sigma)$, i.e.\ that the spanning
set given in the next lemma is actually a basis,
once we have clarified the connection with $U(\g,e)$ in Theorem~\ref{T:gensandPBW}.

\begin{Lemma}\label{L:YPBWspan}
The monomials in the elements
\begin{equation}\label{e:truncYgens}
\begin{array}{c}
\{D_i^{(r)} \mid 1 \le i \le n, 0 < r \le p_i \} \cup \{E_{i,j}^{(r)} \mid 1 \le i  < j \le n, s_{i,j} < r \le s_{i,j} + p_i\} \\
\cup \, \{F_{i,j}^{(r)} \mid 1\le i < j \le n, s_{j, i} < r \le s_{j,i} + p_i\}
\end{array}
\end{equation}
in any fixed order form a spanning set of $Y_{n,l}(\sigma)$.
\end{Lemma}

\begin{proof}
Our proof uses the arguments in the proof of \cite[Lemma~6.1]{BKshift}.
As we are not using the more general parabolic presentations of the Yangian as in that proof, we outline
the arguments required for the convenience of the reader.

During the proof we frequently refer to degree, by which we
always mean filtered degree for the canonical filtration; on occasion
we speak about the total degree of a monomial to make the intended meaning clearer.
Until the final paragraph we use the word monomials to mean unordered
monomials, as this simplifies the exposition.
We frequently use that $\gr Y_n(\sigma)$ is commutative, so for $u \in Y_{n,l}(\sigma)$ of degree $r$
and $v \in Y_{n,l}(\sigma)$ of degree $s$, the commutator $[u,v]$ has degree $ \leq r+s-1$.

For $1 \le k \le n$ and $s \ge 1$, we let:
\begin{itemize}
\item $\Omega_k$ be the set of generators given in \eqref{e:truncYgens} with $i,j \le k$;
\item $\Omega_{k,E}$ be the generators in $\Omega_k$ along with the generators $E_{i,k+1}^{(r)}$
with $1 \le i \le k$ and $s_{i,k+1} < r \le s_{i,k+1}+p_{k+1}$; and
\item $\widehat \Omega_k$ be the set of generators from \eqref{e:pbwgens}
with $i,j \le k$.
\end{itemize}
A key observation for us is:

\smallskip
\noindent
($*$) if $X \in \Omega_k$ with degree $r-s_{k,k+1}$, then $[X,E_k^{s_{k,k+1}+1}]$ can be written
as a linear combination of monomials in $\Omega_{k,E}$ with total degree $r$.

\smallskip
\noindent
This can be checked directly from the relations,
and the definition of $E_{i,j}^{(r)}$ in \eqref{e:Eij}.  A similar statement holds with ``$F$ replacing $E$''.
Further, we have:

\smallskip
\noindent
($\dagger$) if $X \in \Omega_{k,E}$ with degree $r-s_{k+1,k}$, then $[X,F_k^{s_{k+1,k}+1}]$ can be written
as a linear combination of monomials in $\Omega_{k+1} \cup \{\widetilde D_k^{(s)} \mid s = p_k+1,\dots,p_{k+1}\}$
with total degree $r$.

\smallskip
\noindent
Again this is checked directly from the relations, and we note that $\widetilde D_k^{(s)}$ can be
written in terms of $D_k^{(t)}$ for $t \le s$.

We show by induction on $k$ that any element in $\widehat \Omega_k$ of degree $r \ge 0$ can be written
as a linear combination of monomials in the elements of $\Omega_k$ of total degree $r$.

To start the induction we note that the case $k = 1$ is trivial,
because $D_1^{(r)} = 0$ for $r > p_1$ in $Y_{n,l}(\sigma)$.
So suppose inductively we have proved the claim for $\widehat \Omega_k$ and we consider
elements of $\widehat \Omega_{k+1}$.

First consider an element $E_{i,k+1}^{(r)}$ for $r > s_{i,k+1} + p_i$.
For $i<k$,
we use the definition of $E_{i,k+1}^{(r)}$ in \eqref{e:Eij} to
write $E_{i,k+1}^{(r)} = [E_{i,k}^{(r-s_{k,k+1})},E_k^{(s_{k,k+1}+1)}]$.  Using the inductive
hypothesis $E_{i,k}^{(r-s_{k,k+1})}$ can be written
as a sum of monomials in $\Omega_k$ of total degree $r-s_{k,k+1}$.
Now using ($*$) we deduce that $E_{i,k+1}^{(r)}$ can be written as a sum of monomials in $\Omega_{k,E}$
with total degree $r$.
For $i = k$, we have $E_{i,k+1}^{(r)} = E_k^{(r)}$ and we can
use the relation \eqref{e:r4} to write
$$
E_k^{(r)} = \big[D_k^{(r-s_{k,k+1})}, E_k^{(s_{k,k+1}+1)}\big] -  \sum_{t=1}^{r-s_{k,k+1}-1} D_k^{(t)} E_k^{(r-t)}.
$$
The right hand side of the above is an expressions in
elements of $\widehat \Omega_{k+1}$ of degree $r$.
We use property ($*$) to deduce that the first term above
can be written as a sum of monomials in $\Omega_{k,E}$
with total degree $r$.  To deal with the second term we do
an induction on $r$.

We can deal with the elements $F_{i,k+1}^{(r)}$ similarly.

We are left to
consider the elements $D_{k+1}^{(r)}$.
Using \eqref{e:r3}, we write
$$
D_{k+1}^{(r)}  = - \big[E_k^{(r-s_{k+1,k})},F_i^{(s_{k+1,k}+1)}\big]  +
\sum_{t=1}^r D_{i+1}^{(r-t)} \widetilde D_{i}^{(t)},
$$
For the first term on the right hand side we use the above to write $E_k^{r-s_{k+1,k}}$
as a linear combination of monomials in $\Omega_{k,E}$ or total
degree $r-s_{k+1,k}$.  Using ($\dagger$) we rewrite this in terms
of monomials in $\Omega_{k+1} \cup \{\tilde D_k^{(s)} \mid s = p_k+1,\dots,p_{k+1}\}$.
Now we can use the inductive hypothesis to write this as linear combination of
monomials in $\Omega_{k+1}$.  The second term can be dealt with by
induction on $r$.

To finish the proof, we have to observe that for a fixed order on
the elements given in \eqref{e:truncYgens}, an unordered monomial
can be written as a linear combination of ordered monomials.
This is easily done using that $\gr Y_n(\sigma)$ is commutative, an induction on degree,
and what has already been proved.
\end{proof}

Let $Y_{n,l}(\sigma)^{\ab}$ denote the {\em maximal abelian quotient of $Y_{n,l}(\sigma)$}
obtained by factoring out the ideal
generated by all commutators
$\{[u,v] \mid u,v\in Y_{n,l}(\sigma)\}$. So the isomorphism classes of one dimensional representations of
$Y_{n,l}(\sigma)$ are in one-to-one correspondence with maximal ideals of $Y_{n,l}(\sigma)^{\ab}$.
A calculation due to Premet within the proof of \cite[Theorem~3.3]{PrCQ}
shows that $Y_{n,l}(\sigma)^{\ab}$ is generated by a particular
subset of the elements \eqref{e:truncYgens} as stated in the following
lemma.  Although \cite[Theorem~3.3]{PrCQ} is only stated in the characteristic 0 case,
the required calculation works directly from the relations
and we can observe that it is valid in characteristic $p$.

\begin{Lemma}\label{L:1dcal}
The algebra $Y_{n,l}(\sigma)^{\ab}$ is generated by the $l$ elements
$$
\{\dot D_i^{(r)} \mid i=1,...,n, 0 < r \le p_i - p_{i-1}\},
$$
where $\dot D_i^{(r)}$ denotes the image of $D_i^{(r)}$ in $Y_{n,l}(\sigma)^{\ab}$.
\end{Lemma}

\section{$U(\g,e)$ as modular truncated shifted Yangian}

We proceed with the notation
in Section~\ref{S:prelims} and recall that $U(\g,e)$ is the invariant algebra $U(\p)^{\tw(M)}$
for the twisted adjoint action of $M$ on $U(\p)$.
The goal of the current section is to show that $U(\g,e)$ is isomorphic to the truncated
shifted Yangian $Y_{n,l}(\sigma)$ of level $l$.

First we recall some remarkable formulas
from \cite[\S9]{BKshift}
for elements of $U(\p)$, which are actually invariants
for the twisted adjoint action of $M$ as proved in Lemma~\ref{L:Minvariants}.  We
refer also to \cite[\S3.3]{BKrep}, as our notation is closer to the notation used there.
The weight $\eta \in \t^*$ from \eqref{e:eta} is required to define these invariants, and
we note that $\eta$ extends to a character of $\p$.
For $e_{i,j} \in \p$ we define
\begin{equation*}
\te_{i,j} :=
e_{i,j} + \eta(e_{i,j}).
\end{equation*}
Now for $1 \le i,j \le n$, $0 \le x < n$ and $r \ge 1$, we let
\begin{equation}\label{e:invaraintsdef}
T_{i,j;x}^{(r)} := \sum_{s=1}^r (-1)^{r-s} \sum_{\substack{i_1,...,i_s \\ j_1,...,j_s}} (-1)^{|\{t=1,\dots,s-1 \mid \row(j_t) \le x\}|}
\te_{i_1, j_1} \cdots \te_{i_s, j_s} \in U(\p)
\end{equation}
where the sum is taken over all $1 \le i_1,\dots,i_s,j_1,\dots,j_s \le N$ such that
\begin{itemize}
\item[(a)] $\col(j_1)-\col(i_1)+\dots+\col(j_s)-\col(i_s) + s = r$;
\item[(b)] $\col(i_t) \le \col(j_t)$ for each $t = 1,\dots,s$;
\item[(c)] if $\row(j_t) > x$, then $\col(j_t) < \col(i_{t+1})$ for each $t = 1,\dots,s-1$;
\item[(d)] if $\row(j_t) \le x$ then $\col(j_t) \ge \col(i_{t+1})$ for each $t = 1,\dots,s-1$;
\item[(e)] $\row(i_1) = i$, $\row(j_s) = j$;
\item[(f)] $\row(j_t) = \row(i_{t+1})$ for each $t = 1,\dots,s-1$.
\end{itemize}
Now define
\begin{align} \label{e:Dinup}
D_i^{(r)} := T_{i,i;i-1}^{(r)} & \text{ for } 1 \le i \le n,  r > 0   \\ \label{e:Einup}
E_i^{(r)} := T_{i,i+1;i}^{(r)} & \text{ for } 1 \leq i < j \leq n, r > s_{i,j}  \\ \label{e:Finup}
F_i^{(r)} := T_{i+1, i; i}^{(r)} & \text{ for } 1 \leq i < j \leq n, r > s_{j,i} .
\end{align}
These elements are denoted by the same symbols
as the generators of the truncated shifted Yangian and this will be justified
later.  First we prove that they are invariants for the twisted adjoint action of $M$
and thus are elements of $U(\g,e)$.

\begin{Lemma}\label{L:Minvariants}
The elements $D_i^{(r)}$, $E_i^{(r)}$ and $F_i^{(r)}$ of $U(\p)$
defined in \eqref{e:Dinup}, \eqref{e:Einup} and  \eqref{e:Finup} are all invariant under the twisted adjoint action of $M$.
\end{Lemma}

\begin{proof}
Recall from \S\ref{ss:gl} that $\g_\C = \gl_N(\C)$ and $\g_\Z = \gl_N(\Z)$; and we have the
subalgebras $\p_\C$, $\m_\C$, $\p_\Z$ and $\m_\Z$.

Let $X_i^{(r)} \in U(\p)$ be one of the elements
defined by \eqref{e:Dinup}, \eqref{e:Einup} or \eqref{e:Finup}.
Throughout this proof we abuse notation slightly by simultaneously viewing $X_i^{(r)}$ also
as an element of $U(\p_\Z)$ and of $U(\p_\C)$.

According to \cite[Lemma~10.12]{BKshift}, we have that $X_i^{(r)} \in U(\p_\C)$ is invariant under the
twisted adjoint action of $\m_\C$.  This twisted adjoint action is defined by
\begin{equation*}
\tw(x) u := \pr (\ad(x) u)
\end{equation*}
for $x \in \m_\C$ and $u \in U(\p_\C)$.
We see that
the twisted adjoint action of $\m_\C$
exponentiates to give the twisted adjoint action of $M_\C$ as
defined in \eqref{e:twistM}.
Thus we deduce that $X_i^{(r)}$ is an invariant
for the twisted adjoint action of $M_\C$.

Now let $1 \le i,j \le N$ such that $e_{i,j} \in \m_\C$.
Let $t$ be an indeterminate and
consider the homomorphism $U(\g_\C)[t] \to U(\g_\C)[t]$
determined by $e_{k,l} \mapsto e_{k,l} + t \delta_{j,k}e_{i,l} - t\delta_{l,i}e_{k,j} - t^2 \delta_{j,k}\delta_{l,i} e_{i,j}$,
which ``gives the action of $u_{i,j}(t)$ on $e_{k,l}$'' as in \eqref{e:adjoint}.
This preserves the integral form $U(\g_\Z)[t]$ of $U(\g_\C)[t]$ and,
composing with the projection $U(\g_\Z)[t] \to U(\p_\Z)[t]$ along the direct sum
decomposition $U(\g_\Z)[t] = U(\p_\Z)[t] \oplus U(\g_\Z)[t] \{x - \chi(x) \mid x\in \m_\Z\}$,
we obtain a $\Z$-module homomorphism $\psi_{i,j} : U(\p_\Z)[t] \to U(\p_\Z)[t]$. By the observations of the previous paragraph,
$\psi_{i,j}(X_k^{(r)}) - X_k^{(r)} \in (t - s)U(\p_\C)[t]$ for every
$s \in \C$.  It follows that $\psi_{i,j}(X_k^{(r)}) - X_k^{(r)} = 0$ in $U(\p_\C)[t]$.
Note that $\bigcap_{s \in \C} (t - s)U(\p_\C)[t] = 0$
follows from the fact that $U(\p_\C)[t]$ is a free $\C[t]$-module,
and $\bigcap_{s \in \C} (t-s)\C[t] = 0$.

Now consider the equation
\begin{equation*}\label{e:inveq}
\psi_{i,j}(X_k^{(r)}) \otimes 1 - X_k^{(r)} \otimes 1 = 0,
\end{equation*}
valid in $U(\p_\Z)[t] \otimes_\Z \kk \cong U(\p)[t]$.
Examining the image
in $U(\p) \cong U(\p)[t] / (t-s) U(\p)[t]$ for all $s \in \kk$,
we deduce that $X_k^{(r)} \in U(\p)$ is invariant under the twisted adjoint action of the
root subgroup $u_{i,j}(\kk)$.
Hence, $X_k^{(r)} \in U(\p)$ is invariant under the twisted adjoint action of $M$, and this completes the proof.
\end{proof}

We define elements $E_{i,j}^{(r)} \in U(\p)$ for $1 \le i < j \le n$ and $r > s_{i,j}$ from the
expressions for $E_i^{(r)} \in U(\p)$ given in \eqref{e:Einup}
and the recursive formula in \eqref{e:Eij}; we define $F_{i,j}^{(r)} \in U(\p)$
similarly.  From these definitions and Lemma~\ref{L:Minvariants}, we have
that these $E_{i,j}^{(r)}$ and $F_{i,j}^{(r)}$ are actually elements of $U(\g,e)$.
For the next lemma we recall the basis for $\g^e$ from Lemma~\ref{L:centraliserbasis},
and the notation for elements $\be^\ba$ in $U(\p)$
given in \eqref{e:e^a}.

\begin{Lemma}\label{L:lineartops}
$ $
\begin{itemize}
\item[(a)] For $1 \le i \le n$, and $1 \le r \le p_i$, we have
$D_i^{(r)} = (-1)^{r-1} c_{i,i}^{(r)} + u$, where $u$ is a linear combination of
terms $\be^\ba$ satisfying either $|\ba|_e = r$ and $|\ba| > 1$, or  $|\ba|_e < r$.
\item[(b)] For $1 \le i< j \le n$, and $s_{i,j} < r \le p_i+ s_{i,j}$, we have
$E_{i,j}^{(r)} = (-1)^{r-1} c_{i,j}^{(r)} + u$, where $u$ is a linear combination of
terms $\be^\ba$ satisfying either $|\ba|_e = r$ and $|\ba| > 1$, or  $|\ba|_e < r$.
\item[(c)] For $1 \le i < j \le n$, and $s_{j,i} < r \le p_i+ s_{j,i}$, we have
$F_{i,j}^{(r)} = (-1)^{r-1} c_{j,i}^{(r)} + u$, where $u$ is a linear combination of
terms $\be^\ba$ satisfying either $|\ba|_e = r$ and $|\ba| > 1$, or  $|\ba|_e < r$.
\item[(d)] The monomials in
$$
\begin{array}{c}
\{D_i^{(r)} \mid 1 \le i \le n, 1 \le r \le p_i\} \cup \{E_{i,j}^{(r)} \mid
1 \le i< j \le n, s_{i,j} < r \le p_i+ s_{i,j}\}  \\ \cup
\{F_{i,j}^{(r)} \mid
1 \le i< j \le n, s_{j,i} < r \le p_i+ s_{j,i}\}
\end{array}
$$
taken in any fixed order form a basis of $U(\g,e)$.
\end{itemize}
\end{Lemma}

\begin{proof}
We begin by proving (a).  Consider the expression
given for $D_i^{(r)}$ given by \eqref{e:invaraintsdef} and \eqref{e:Dinup}.
We can verify that the terms for $s > 1$ are a linear combination of
terms $\be^\ba$ satisfying $|\ba|_e = r$ and $|\ba| > 1$, or  $|\ba|_e < r$.
So we are left to show that the $s=1$ part is precisely
$(-1)^{r-1} c_{i,i}^{(r)}$, and this follows directly from the definitions.

The cases of (b) and (c) for $j = i+1$ are proved similarly to (a).
Then using the definitions in \eqref{e:Eij} and \eqref{e:Fij} along
with Lemma~\ref{L:centraliserbasis}(b), and \eqref{e:commTheta}, we deduce
the statement for all $i$ and $j$.

Part (d) is now an immediate consequence of
Lemma~\ref{L:centraliserbasis}(a) and Proposition~\ref{P:PBW}.
\end{proof}

We are now in a position to
prove the main result of the section, showing that $Y_{n,l}(\sigma)$ is
isomorphic to $U(\g,e)$, which is a modular analogue of \cite[Theorem~10.1]{BKshift}

\begin{Theorem}
\label{T:gensandPBW}
The map from $Y_{n,l}(\sigma)$ to $U(\g,e)$ determined by sending each element of $Y_{n,l}(\sigma)$ in
$$
\begin{array}{c}
\{D_i^{(r)} \mid 1 \le i \le n, r \ge 1\} \cup
\{E_i^{(r)} \mid 1 \le i < n, r > s_{i,i+1} \}
\cup \{F_i^{(r)} \mid  1 \le i < n, r > s_{i+1,i} \}
\end{array}
$$
to the element of $U(\g,e)$ denoted by the same symbol
defines an isomorphism
$$
Y_{n,l}(\sigma) \isoto U(\g,e).
$$
\end{Theorem}

\begin{proof}
Consider the elements $D_i^{(r)}, E_i^{(r)}, F_i^{(r)}$ defined in $U(\p_\C)$ by formulas
\eqref{e:invaraintsdef}, \eqref{e:Dinup}, \eqref{e:Einup}, \eqref{e:Finup}.
It follows from \cite[Theorem~10.1]{BKshift} that these elements satisfy relations \eqref{e:r2}--\eqref{e:r13}.
Also they satisfy \eqref{e:rel15} and \eqref{e:rel16} as these follow, over $\C$ from \cite[(2.14), (2.15)]{BKshift}.
Since all of these relations have integral coefficients they hold in $U(\p_\Z)$, and
thus also in $U(\p) \iso U(\p_\Z) \otimes \kk$.  In addition it is clear from the definition of $D_1^{(r)} \in U(\g,e)$ given in
\eqref{e:invaraintsdef} and \eqref{e:Dinup}, that $D_1^{(r)} = 0$ for $r > p_1$.

Hence, the map described in the statement does give a homomorphism
$Y_{n,l}(\sigma) \to U(\g,e)$.  By Lemma~\ref{L:lineartops}(d), we see that this homomorphism
is surjective, and  using Lemma~\ref{L:YPBWspan}, we deduce that it is injective.
\end{proof}

As mentioned before Lemma~\ref{L:YPBWspan}, we are now able to deduce a PBW theorem for
$Y_{n,l}(\sigma)$, as given in \cite[Corollary~6.3]{BKshift}.  This says that the monomials in the elements given in \eqref{e:truncYgens}
taken in any fixed order form a basis of $Y_{n,l}(\sigma)$, and follows immediately from
Lemma~\ref{L:lineartops} and Theorem~\ref{T:gensandPBW}.

\section{1-dimensional modules for $U(\g,e)$}

We follow similar methods to those in \cite[Section 2]{BrM} to classify the 1-dimensional
modules for $U(\g,e)$.  Before stating this in Theorem~\ref{T:1dW} we require some
notation, and we also use the generators and relations for $U(\g,e)$ given by
Theorem~\ref{T:gensandPBW} to make some initial deductions about 1-dimensional
modules for $U(\g,e)$.

It follows from \eqref{e:r2}
and Lemma~\ref{L:lineartops} that $\{D_i^{(r)} \mid i = 1,\dots,n, 1 \le r \le p_i\}$
generates a subalgebra of $U(\g,e)$ isomorphic to a polynomial algebra in $N$ variables; we denote this subalgebra
by $U(\g,e)^0$.

Let $A \in \Tab_\kk(\pi)$.  For $i=1,\dots,n$, we write $a_{i,1},\dots,a_{i,p_i}$ for the entries in the $i$th row
of $A$ from left to right.  We define the 1-dimensional $U(\g,e)^0$-module $\kk_{\bar A}$
by saying that $D_i^{(r)}$ acts on $\kk_A$ by $e_r(a_{i,1}+i,\dots,a_{i,p_i}+i)$; here
$e_r(x_1,\dots,x_{p_i})$
is the $r$th elementary symmetric polynomial in the indeterminates $x_1,\dots,x_{p_i}$.
It is clear that $\kk_{\bar A}$ depends only on the row equivalence class of $A$.
We note that given $b_1,\dots,b_{p_i} \in \kk$, finding $c_{i,1},\dots,c_{i,p_i}$
such that $e_r(c_1,\dots,c_{p_i}) = b_r$ for each $r$ is equivalent to finding
solutions of the polynomial $t^{p_i}-b_1t^{p_i-1}+\dots+(-1)^{p_i-1}b_{p_i-1}t+ (-1)^{p_i}b_{p_i}$.
Therefore, since $\kk$ is algebraically closed, we see that any 1-dimensional
$U(\g,e)^0$-module is isomorphic to $\kk_{\bar A}$ for some $A \in \Tab_\kk(\pi)$.
Thus we see that the restriction of any 1-dimensional $U(\g,e)$-module to $U(\g,e)^0$ is isomorphic
to $\kk_{\bar A}$ for some $A \in \Tab_\kk(\pi)$, and that such $A$ is defined up to
row equivalence.  Using the relations \eqref{e:r4} and \eqref{e:r5} for $r=1$,
we see that the generators $E_i^{(s)}$ and $F_i^{(s)}$ act as 0 on any 1-dimensional
$U(\g,e)$-module, for all $i$ and $s$.
If the action of $U(\g,e)^0$ on $\kk_{\bar A}$ can be extended to a $U(\g,e)$-module, on which
$E_i^{(s)}$ and $F_i^{(s)}$ act as 0, then we denote
this module by $\widetilde \kk_{\bar A}$.  Our goal is thus to determine when
$\widetilde \kk_{\bar A}$ exists, and this is achieved in the following theorem.

\begin{Theorem} \label{T:1dW}
Let $A \in \Tab_\kk(\pi)$.
There is a 1-dimensional $U(\g,e)$-module $\widetilde \kk_{\bar A}$,
which extends the action of $U(\g,e)^0$ on $\kk_{\bar A}$ if and only
if $A$ is row equivalent to a column connected tableau.
\end{Theorem}

\begin{proof}
First let $A \in \Tab_\kk(\pi)$ be column connected, with entries
in the $i$th row labelled
$a_{i,1},\dots,a_{i,p_i}$ for $i=1,\dots,n$.  Recall the 1-dimensional
$U(\h)$-module $\widetilde \kk_A$ defined in \S\ref{ss:modules}; this can be inflated
to a $U(\p)$ module, which we also denote by $\widetilde \kk_A$.
Consider the action of the explicit
elements $D_i^{(r)} \in U(\p)$ given in \eqref{e:invaraintsdef} and \eqref{e:Dinup}
on the module $\widetilde \kk_A$.
The only summands in the expression for $D_i^{(r)}$ that do not act as zero on $\widetilde \kk_A$
are those which are products $\te_{i_1,j_1} \cdots \te_{i_s,j_s}$ such that
$i_1=j_1,\dots,i_s=j_s$: terms of this form only occur for $s=r$ and their sum is precisely
$e_r(\te_{i_1,j_1},\dots,\te_{ir_,j_r})$.  The $\t$-weight of $\widetilde \kk_A$ is $\lambda_A - \widetilde \rho$,
and we have $\widetilde \rho = \eta + \rho_\h$ by \eqref{e:tilderho}.  Thus we see that each $\te_{i,i}$
acts on $\widetilde \kk_A$ by $(\lambda_A+\rho_\h)(e_{i,i})$, because
of the shift of $\eta$ in the definition of $\te_{i,j}$.  Combining all of these
observations shows that the action of $D_i^{(r)}$ on $\widetilde \kk_A$
is given by $e_r(a_{i,1}+i,\dots,a_{i,p_i}+i)$.
This proves that $\widetilde \kk_{\bar A}$ exists, under the assumption that $A$ is
column connected.

We move on to prove that $\widetilde \kk_{\bar A}$ exists only if
$A$ is column connected.  To do this first note that the action $U(\g,e)$ on any 1-dimensional
factors to an action of the abelianization $U(\g,e)^{\ab}$ of $U(\g,e)$.
By Lemma~\ref{L:1dcal} and Theorem~\ref{T:gensandPBW} we know that
$U(\g,e)^{\ab}$ is generated by the images of the $l$ elements
$$
\{D_i^{(r)} \mid i=1,\dots,n, 0 < r \le p_i - p_{i-1}\}.
$$
So any 1-dimensional $U(\g,e)$-module is determined uniquely
by the action of these elements.
Thus to show that any one dimensional $U(\g,e)$-module is of the form
$\widetilde \kk_{\bar A}$ for some column connected $A \in \Tab_\kk(\pi)$, it
suffices to show that for any set
$$
\{a_i^{(r)} \mid i=1,\dots,n, 0 < r \le p_i - p_{i-1}\}.
$$
where $a_i^{(r)} \in \kk$, there is a column connected
$A \in \Tab_\kk(\pi)$ such that the action of $D_i^{(r)}$
on $\widetilde \kk_{\bar A}$ is given by $a_i^{(r)}$ for $i=1,\dots,n$ and $0 \le r \le p_i - p_{i-1}$.
This is proved ``over the complex numbers'' at the
end of \cite[Section 2]{BrM}, and depends crucially on \cite[Lemma~2.6]{BrM}.
It can be observed that the proof of this lemma is also valid
over an algebraically closed field of characteristic $p$.
From this we can deduce that all 1-dimensional
$U(\g,e)$-modules are of the form $\tilde \kk_{\bar A}$ for
some column connected $A \in \Tab_\kk(\pi)$.
\end{proof}

As mentioned after the statement of Lemma~\ref{L:1dcal}, the abelianization of
$Y_{n,l}(\sigma)$ is actually a polynomial algebra on the generators given in that
lemma.  This can now be deduced immediately from the proof of Theorem~\ref{T:1dW}, which
shows that there are 1-dimensional modules for $Y_{n,l}(\sigma)$ on which these generators
can act by arbitrary elements of $\kk$.

\section{1-dimensional modules for $U_0(\g,e)$}

From Theorem~\ref{T:1dW} we have a classification of 1-dimensional
$U(\g,e)$-modules given by the modules $\widetilde \kk_{\bar A}$ for $A \in \Tab_\kk(\pi)$
ranging over a set of representatives of row equivalence
classes of column connected tableaux.
Our next theorem determines for which of these 1-dimensional modules the action
of $U(\g,e)$ factors
through the quotient $U(\g,e) \onto U_0(\g,e)$ to give a 1-dimensional $U_0(\g,e)$-module.
Therefore, we obtain a classification of the 1-dimensional $U_0(\g,e)$-modules.

\begin{Theorem} \label{T:1dsforrestw}
Let $A \in \Tab_\kk(\pi)$ be column connected.  Then $\widetilde \kk_{\bar A}$
factors to a module for $U_0(\g,e)$ if and only if $A \in \Tab_{\F_p}(\pi)$.
\end{Theorem}

\begin{proof}
Let $A \in \Tab_\kk(\pi)$ be column connected with entries labelled
$a_1,\dots,a_N$ as usual.  We see that $\widetilde{\kk}_A$ factors
through $U_{\psi_A}(\p)$, where $\psi_A \in \p^*$ is defined by
$\psi_A(e_{i,j}) = 0$ for $i \ne j$, and $\psi_A(e_{i,i}) = a_i^p - a_i$.
Therefore, $\widetilde \kk_{\bar A}$ is a module for
the reduced $W$-algebra $U_{\psi_A}(\g,e)$ as defined in \S\ref{ss:walg}.
From the discussion after Lemma~\ref{L:samered},  we see that
$\widetilde \kk_{\bar A}$ factors to a module for $U_0(\g,e)$ if and only if
$0$ and $\psi_A$ are conjugate under the twisted $M$-action.

We next show that $\psi_A$ is conjugate to 0 under the twisted $M$-action only
if $\psi_A = 0$.  To do this note that under the identification $\p^* \iso e + \m^{\perp}$
we have that $0$ corresponds to $e$ and $\psi_A$ corresponds to an element $e + \diag(a_1^p-a_1,\dots,a_N^p-a_N)$,
where we recall that $\diag(d_1,\dots,d_N)$ denotes the diagonal matrix with $i$th entry $d_i$.
We have $e + \diag(a_1^p-a_1,\dots,a_N^p-a_N) \in \t$ is nilpotent only if $\diag(a_1^p-a_1,\dots,a_N^p-a_N) =0$.
Therefore, $e$ is not in the same $M$-orbit as $e + \diag(a_1^p-a_1,\dots,a_N^p-a_N)$ unless
$a_i^p-a_i = 0$ for all $i$.

Hence, we deduce that $\widetilde \kk_{\bar A}$ factors to
a module for $U_0(\g,e)$ if only if $a_i^p-a_i = 0$ for all $i$.
This is case if and only if $a_i \in \F_p$ for all $i$, so that
$A \in \Tab_{\F_p}(\pi)$.
\end{proof}

\section{Minimal dimensional modules for $U_\chi(\g)$}

Armed with Premet's equivalence (Theorem~\ref{T:premequiv}), Theorem~\ref{T:1dsimplespara}
and Theorem~\ref{T:1dsforrestw}, we are ready to prove our main theorem.

\begin{proof}[Proof of Theorem~\ref{T:main}]
Let $c_\pi$ the number
of row equivalence classes in $\Tab_{\F_p}(\pi)$ containing a column connected tableau.
By Theorem~\ref{T:1dsforrestw}, we know that the number of isomorphism classes of
$1$-dimensional modules for $U_0(\g,e)$-modules is $c_\pi$.  Thus by Theorem~\ref{T:premequiv}
the number of minimal dimensional $U_\chi(\g)$-modules is $c_\pi$.

Given column connected $A \in \Tab_{\F_p}(\pi)$,
we have that $L_\chi(A)$ is $p^{d_\chi}$-dimensional by Theorem~\ref{T:1dsimplespara}.
Also up to isomorphism $L_\chi(A)$ depends only on the row equivalence class of $A$.
Therefore, the modules $L_\chi(A)$ for $A$ ranging over a set of row equivalence
classes of column connected tableaux in $\Tab_{\F_p}(\pi)$ give all $c_\pi$ isomorphism
classes of minimal dimensional modules.
\end{proof}

Now that Theorem~\ref{T:main} is proved, Corollary~\ref{C:moeglin} follows, as explained in the
introduction.

\begin{Remark} \label{R:corresp}
It is interesting to know the bijection given by Premet's equivalence between the sets
of isomorphism classes of the 1-dimensional
$U_0(\g,e)$-modules $\widetilde \kk_{\bar A}$ and those of the minimal dimensional
$U_\chi(\g)$-modules $L_\chi(A)$, as $A$ ranges over a set of representatives
of row equivalence classes of column connected tableau in $\Tab_{\F_p}(\pi)$.
It turns out that this bijection sends $\widetilde \kk_{\bar A}$ to $L_\chi(A)$
and we briefly outline some steps that can be used to verify this.

Use the fact that $L_\chi(A) \iso N_\chi(A) = U_\chi(\g) \otimes_{U_0(\p)} \bar \kk_A$
as is given in Theorem~\ref{T:1dsimplespara}.  Then show that $N_\chi(A)^{\m_\chi} \iso \widetilde \kk_{\bar A}$
using the following arguments; we recall here that $N_\chi(A)^{\m_\chi}$ is defined in \eqref{e:premequiv2}.

Consider the dual $N_\chi(A)^*$ viewed as a right module for $U_\chi(\g)$.
We observe that $\lambda_A - \bar \rho - \beta = \lambda - \tilde \rho$ is the
weight of a $1$-dimensional right $U_0(\h)$-module, which we denote by $\kk_{\lambda_A - \bar \rho - \beta}$.
Using that any $U_\chi(\g)$-module is free as a $U_\chi(\m)$-module, it can
be proved that $N_\chi(A)^* \iso \kk_{\lambda_A - \bar \rho - \beta} \otimes_{U_0(\p)} U_\chi(\g)$.
Note that it is more natural to consider weight $\lambda_A - \bar \rho + (p-1)\beta$ to prove this isomorphism,
and use that this is the $\t$-weight of $\prod_{\alpha \in \Phi(-)} e_\alpha^{p-1} \bar 1_A \in N_\chi(A)$, where $\bar 1_A$
is the generator of $\bar \kk_A$.

Next consider the {\em Whittaker coinvariants} of $N_\chi(A)^*$.  This is defined by $N_\chi(A)^*/N_\chi(A)^*\m_\chi$,
and is a right module for $U(\g,e)$.  It is quite straightforward to show that the Whittaker coinvariants
of $\kk_{\lambda_A - \bar \rho - \beta} \otimes_{U_0(\p)} U_\chi(\g)$ is isomorphic to the
restriction of right $U_0(\p)$-module $\kk_{\lambda_A - \bar \rho - \beta}$ to $U(\g,e)$, so the same is true
for the Whittaker coinvariants of $N_\chi(A)^*$.
Standard arguments show that $N_\chi(A)^{\m_\chi} \iso (N_\chi(A)^*/\m_\chi N_\chi(A)^*)^*$.
Then it can be deduced that $N_\chi(A)^{\m_\chi}$ is isomorphic to the restriction of the left $U_0(\p)$-module
$\kk_{\lambda_A - \bar \rho - \beta}$ to $U_0(\g,e)$.

It just remains to just use the fact that $\bar \rho + \beta = \widetilde \rho$ to deduce that
$N_\chi(A)^{\m_\chi} \iso \widetilde \kk_{\bar A}$.
\end{Remark}

\end{document}